\documentclass[11pt,a4paper]{article}
\usepackage{graphicx}
\usepackage{amsmath,amsfonts,amsthm}
\usepackage{xfrac,esint,mathrsfs,color}
\usepackage[utf8]{inputenc}
\usepackage[colorlinks=true]{hyperref}
\usepackage{amssymb}

\def\calM{\mathcal{M}}

\newcommand\eps{\varepsilon}

\newcommand\diam{\mathrm{diam\,}}

\newcommand\dist{\mathrm{dist}}

\newcommand\R{\mathbb{R}}
\newcommand\N{\mathbb{N}}

\newcommand\calH{\mathcal{H}}
\newcommand\calN{\mathcal{N}}
\newcommand\calF{\mathcal{F}}

\newcommand\calA{\mathcal{A}}
\newcommand\calG{\mathcal{G}}
\newcommand\calB{\mathcal{B}}
\newcommand\calL{\mathcal{L}}

\providecommand{\ds}{\, \mathrm{d} s}
\providecommand{\dt}{\, \mathrm{d} t}
\providecommand{\dx}{\, \mathrm{d} x_1}
\providecommand{\dy}{\, \mathrm{d} x_2}

\providecommand{\dxy}{\, \mathrm{d} x}
\providecommand{\dyy}{\, \mathrm{d} y}
\providecommand{\dH}{\,  \mathrm{d} \calH^1}
\newcommand{\LM}[1]{\hbox{\vrule width.2pt \vbox to#1pt{\vfill \hrule width#1pt height.2pt}}}
\newcommand{\LL}{{\mathchoice{\,\LM7\,}{\,\LM7\,}{\,\LM5\,}{\,\LM{3.35}\,}}}
\newtheorem{theorem}{Theorem}[section]

\newtheorem{proposition}[theorem]{Proposition}
\newtheorem{lemma}[theorem]{Lemma}
\newtheorem{example}[theorem]{Example}

\definecolor{green}{rgb}{0,.5,0}

\numberwithin{equation}{section}
\newcounter{Nummer}

\begin{document}
\begin{center}
  {\Large
  Deformation concentration for
martensitic microstructures
in the limit of low volume fraction
  }\\[5mm]
{December 22, 2015}\\[5mm]
Sergio Conti, Johannes Diermeier, and Barbara Zwicknagl\\[2mm]
{\em 
 Institut f\"ur Angewandte Mathematik,
Universit\"at Bonn\\ 53115 Bonn, Germany}\\[1mm]
    \begin{minipage}[c]{0.8\textwidth}
    \end{minipage}
\end{center}
\begin{abstract}
We consider a singularly-perturbed nonconvex energy functional which arises in the study of microstructures in shape memory alloys. The  scaling law for the minimal energy predicts a transition from 
a parameter regime in which uniform structures are favored, to a regime in which the formation of fine patterns is expected. We focus on the transition regime and derive the reduced model in the 
sense of $\Gamma$-convergence. The limit functional turns out to be similar to the Mumford-Shah functional with additional constraints on the jump set of admissible functions.
One key ingredient in the proof is an approximation result for $SBV^p$ functions whose jump
sets have a prescribed orientation.
\end{abstract}
\section{Introduction}
In this paper, we consider for $\theta\in(0,1/2]$, $\eps>0$ and $p\in (1,\infty)$ the energy functional $I_{\eps,\theta}^p:\calB\to[0,\infty)$, 
\begin{equation}\label{eq:unrescaled}
I_{\eps,\theta}^{p}(v):=
\int_{(0,1)^2}|\partial_1v|^p+\min\left\{|\partial_2v+\theta|^p,\ |\partial_2 v-(1-\theta)|^p\right\}\dxy+\eps|D^2v|((0,1)^2),
\end{equation}
where 
\begin{eqnarray}
\label{eq:admunr}
\calB:=\{v\in W^{1,p}((0,1)^2):\ \partial_1v,\ \partial_2v\in BV((0,1)^2),\ v(0,\cdot)=0\}.
\end{eqnarray}
Here we use the short-hand notation $\partial_jv:=\frac{\partial}{\partial{x_j}}v$ for $j=1,2$. It turns out that for fixed $p$, there are two scaling regimes for the minimal energy. 
If the coefficient {$\eps$} of the higher-order term  is large, on a scale set by $\theta$, then low-energy maps $v$ are approximately 
constant, and the optimal energy is of order $\theta^p$. If instead $\eps$ is very small
fine structures arise, with $\partial_2 v$ oscillating between $-\theta$ and $1-\theta$, on a scale which refines
close to the ${{\{}}x_1=0{{\}}}$ boundary.
Precisely, we 
have the following result. 
\begin{theorem}\label{th:scaling}
For any $p\in (1,\infty)$ there exists $c>0$ such that for all $\theta\in(0,1/2]$ and all $\eps>0$,
\begin{eqnarray}\label{eq:sl}
\frac1c\theta^p\min\{1,\ (\eps/\theta^p)^{p/(p+1)}\}\leq \inf_{\calB} I_{\theta,\eps}^p\leq c\theta^p\min\{1,\  (\eps/\theta^p)^{p/(p+1)}\}.
\end{eqnarray}
\end{theorem}
The proof of this scaling law \eqref{eq:sl} is fairly standard and well known in the case $p=2$ (see \cite{kohn-mueller:92,kohn-mueller:94,conti:06,diermeier:10,zwicknagl:14}), and we provide it in 
the appendix for general $p$.
In this paper, we focus on the transition regime, in which $\eps/\theta^p=:\sigma\in (0,\infty)$ is fixed. In this case, $\inf I_{\sigma\theta^p,\theta}^p\sim \theta^p$, and, after rescaling 
$v(x_1,x_2):=\theta(u(x_1,x_2)-x_2)$ and  $E_{\sigma,\theta}^p(u):=I_{\sigma\theta^p,\theta}^p(v)/\theta^p$, we are left to study
\begin{eqnarray}
E_{\sigma,\theta}^p(u):=\begin{cases}
\displaystyle
\int_{(0,1)^2}|\partial_1u|^p+\min\left\{|\partial_2u|^p,\ |\partial_2 u-\frac{1}{\theta}|^p\right\}\dxy&\\
\hskip5cm +\sigma\theta|D^2u|((0,1)^2)&\text{\ if\ }u\in\calA,\\
+\infty&\text{\ otherwise},
\end{cases}\label{eq:funcrecaled}
\end{eqnarray}
where 
\begin{eqnarray}
\calA:=\{ u\in W^{1,p}((0,1)^2):\ \partial_1 u,\ \partial_2 u\in BV((0,1)^2),\  u(0,x_2)=x_2\}.
\end{eqnarray}
For small $\theta$ the deformation concentrates on a {{set of}} small {{volume}}. In particular, $\partial_2 u$ becomes of order
$1/\theta$ in a small region of order $\theta$. The length of the boundary of  this 
exceptional set is controlled by the second-order term, since in going across it the gradient
$Du$ has two jumps of order $1/\theta$. Asymptotically $u$ approaches locally an $SBV$ function, with the additional property
that the singular part of the gradient has a specific orientation. 

Our main result is the $\Gamma$-limit of $E_{\sigma,\theta}^p$ as $\theta\to 0$. Precisely, 
we set
\begin{eqnarray}\label{eq:limitfunc}
  E_\sigma^p(u):=\begin{cases}\displaystyle
  \int_{{(0,1)^2}} (|\partial_1 u|^p +|\partial_2u|^p)\dxy + 2\sigma \calH^1(J_u)&\text{\ if\ }u\in \overline{SBV}^p_y,\\
  +\infty&\text{\ otherwise},
  \end{cases}
\end{eqnarray}
where 
\begin{eqnarray}\nonumber
\overline{SBV}^p_y:=\{u\in SBV_{\text{loc}}((0,1)^2): \hspace{-.5cm}&&\nabla u \in L^p((0,1)^2;\R^2),
\hskip2mm u(0,x_2)=x_2, \\
&&
\text{and } [u]\nu_u\in [0,\infty)e_2 {\ }\calH^1-\text{a.e.}\}.
\label{eq:admlimit}
\end{eqnarray}
We remark that for any $u\in \overline{SBV}^p_y$ one has
$|Du|((0,1)\times(\delta,1-\delta))<\infty$
for all $\delta\in (0,1/2)$, see Lemma \ref{lemmasbvpy} below.
We prove the following theorem.
\begin{theorem}\label{th:gammalimit}
\begin{itemize}
\item[(i)] {\em Compactness. }
Suppose that $\theta_k\to 0$, and let $u_k\in\mathcal{A}$ such that $E_{\sigma,\theta_k}^p(u_k)\leq M$ for some $M>0$. Then there exists a subsequence (not relabeled) and $u\in 
\overline{SBV}_y^p$ such that 
$u_k\to u$ in $L^1((0,1)^2)$ and
$u_k\overset{\ast}{\rightharpoonup} u$ in $BV((0,1)\times(\delta,1-\delta))$ for all $\delta\in (0,1/2)$. 
\item[(ii)] {\em Lower bound. }Suppose that $\theta_k\to 0$. Let $u_k\in\cal{A}$ and
assume that $u_k\to u$ in $L^1((0,1)^2)$
 for some $u\in \overline{SBV}^p_{y}$.
 Then $E_\sigma^p(u)\leq\liminf_{k\to\infty}E_{\sigma\theta_k^p,\theta_k}^p(u_k)$.
\item[(iii)] {\em Upper bound. }Let $\theta_k\to 0$ and $u\in \overline{SBV}_y^p$. Then there exist $u_k\in\calA$ such that 
 $u_k\to u$ in $L^1((0,1)^2)$ and 
$E_\sigma^p(u)=\limsup_{k\to\infty}E_{\sigma\theta^p_{{k}},\theta_k}^p(u_k)$. 
\end{itemize}
\end{theorem}
By (i) the sequence in (iii) obeys also 
$u_k\overset{\ast}{\rightharpoonup} u$ in $BV((0,1)\times(\delta,1-\delta))$ for all $\delta\in (0,1/2)$.
The result {{of Theorem \ref{th:gammalimit}}} has been announced in \cite{diermeier:15} for the case $p=2$ and is part of Johannes Diermeier's Ph.D. thesis (see \cite{diermeier:16}). We focus here on the two-dimensional situation, and 
refer to \cite{diermeier:16} for an analogous result in one dimension. For ease of notation, we consider the functional on the unit
square $(0,1)^2$, but we point out that our analysis can 
be easily adapted to treat more general domains $\Omega\subset\R^2$.

The limit functional $E_\sigma^p$ bears similarities with the Mumford-Shah functional (see \cite{DeGiorgi-et-al:89,Mumford-Shah:89}). The main difference lies in the constraints on the jump sets of 
admissible functions, see \eqref{eq:admlimit}. This introduces technical difficulties in the proof of the upper bound, for which we follow the general strategy to use density of functions with a 
simpler structure. More precisely, we provide explicit constructions of recovery sequences only for functions whose jump sets consist of finitely many segments and which are smooth away from their 
jump sets, and prove an accompanying density result to deal with general functions in $\overline{SBV}_y^p$. To the best of our knowledge, there are no approximation results in the literature that 
respect the constraints for functions in $\overline{SBV}_y^p$ (see e.g. \cite{cortesani-toader:1999, dibos-sere:1997} and the references given there).

The motivation for our analysis comes from the mathematical study of martensitic phase transitions, where the functional $I_{\theta,\eps}^2$ as defined in \eqref{eq:unrescaled} arises in the study of 
microstructures near interfaces (see  \cite{kohn-mueller:92,kohn-mueller:94}). Here, $(0,1)^2$ represents a martensitic region
meeting rigid austenite at an interface $\{x_1=0\}$. The first two terms in 
the definition of $I_{\eps,\theta}^p$ model the elastic energy, and the last term can be interpreted as an interfacial energy
between different variants of martensite; 
the coefficient $\eps$ represents a typical interfacial energy 
per unit length. The preferred gradients $(0,-\theta)^T$ and $(0,1-\theta)^T$ correspond to two variants of martensite, and the parameter $\theta\in(0,1/2]$ measures compatibility between the 
austenite and the martensite phases: For $\theta=0$, austenite and martensite are compatible in the sense that $v_c:=0$ satisfies $I_{\eps,\theta}^p(v_c)=\min I_{\eps,\theta}^p=0$, while for 
$\theta>0$, we have $\inf I_{\eps,\theta}^p>0$. It has been found in experiments that
compatibility between the phases is closely related to the width of the hysteresis loop 
accompanying the phase transformation (see \cite{james-zhang:05,cui-et-al:06,zhang:07,zarnetta-et-al,louie-et-al,Sriva-et-al,bechthold-et-el:12}). This theory of hysteresis predicts that the energy 
barrier $\inf I_{\eps,\theta}^p$ plays a major role for reversibility of the phase transformation {{(see \cite{zjm:09,chluba-et-al:15})}}.

As pointed out before, the scaling law for the minimal energy \eqref{eq:sl} is well-understood for the physically relevant case $p=2$ (see 
\cite{kohn-mueller:92,kohn-mueller:94,conti:06,diermeier:10,zwicknagl:14,conti-zwicknagl:15}). These studies suggest a transition between uniform structures and the formation of 
microstructures. Some of the results have been extended also to the vector-valued case (see 
\cite{chan:13,chan-conti:14,chan-conti:14-1,diermeier:13,CapellaOtto2012,CapellaOtto2009,knuepfer-kohn:11,knuepfer-kohn-otto:13}). Similar phenomena have been found for a variety of 
variational models, with applications including pattern formation in ferromagnets (see  \cite{choksi-kohn:98,choksi-et-al:98,otto-viehmann:10,knuepfer-muratov:11}), in type-1-superconductors {{(see 
\cite{choksi-et-al:04,choksi-et-al:08,conti-et-al:15})}}, and in thin compressed films (see \cite{BCDM00,JinSternberg2,JinSternberg1,belgacem-et-al:02,bella-kohn:14}).
{In many situations, peculiar structures arise in the limit of small volume fraction, which corresponds 
to the limit $\theta\to0$ considered in this paper \cite{choksi-et-al:04,choksi-et-al:08,diermeier:10,zwicknagl:14,conti-zwicknagl:15}.}

Typical test functions that can be used to prove the second inequality in \eqref{eq:sl} are sketched in Figure \ref{fig:scaling}. 
\begin{figure}
\centerline{ \includegraphics[height=5cm]{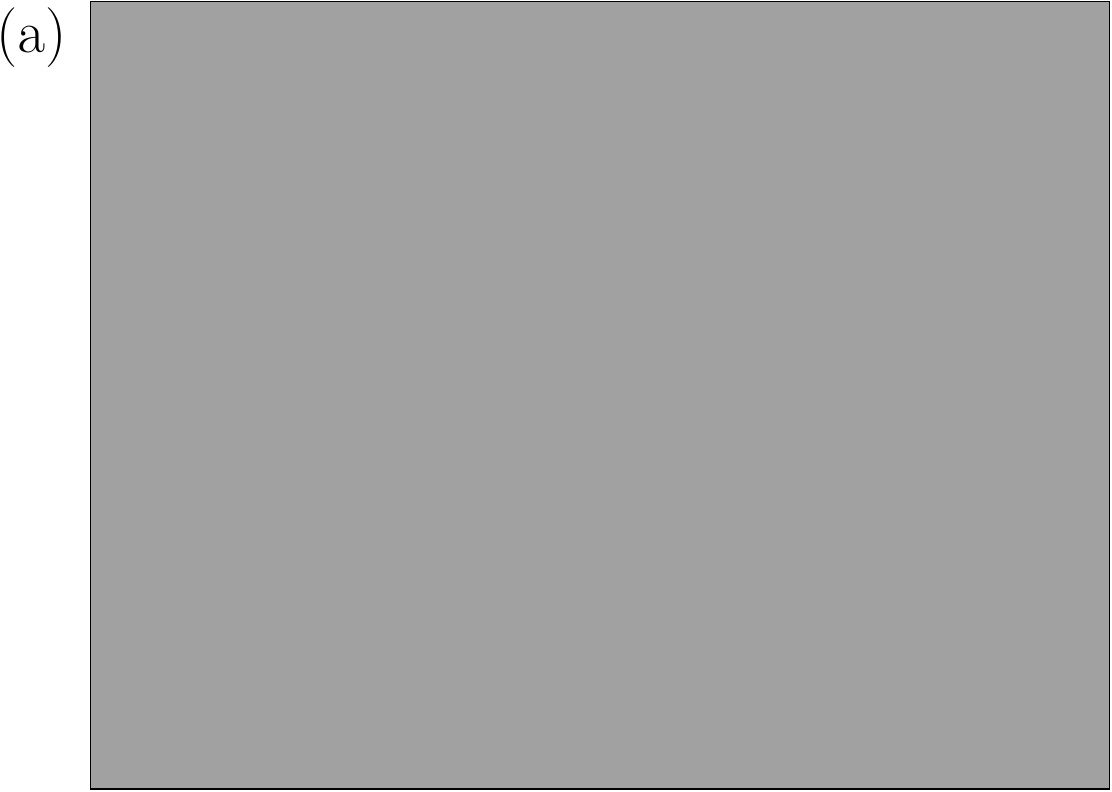}\hskip.8cm
 \includegraphics[height=5cm]{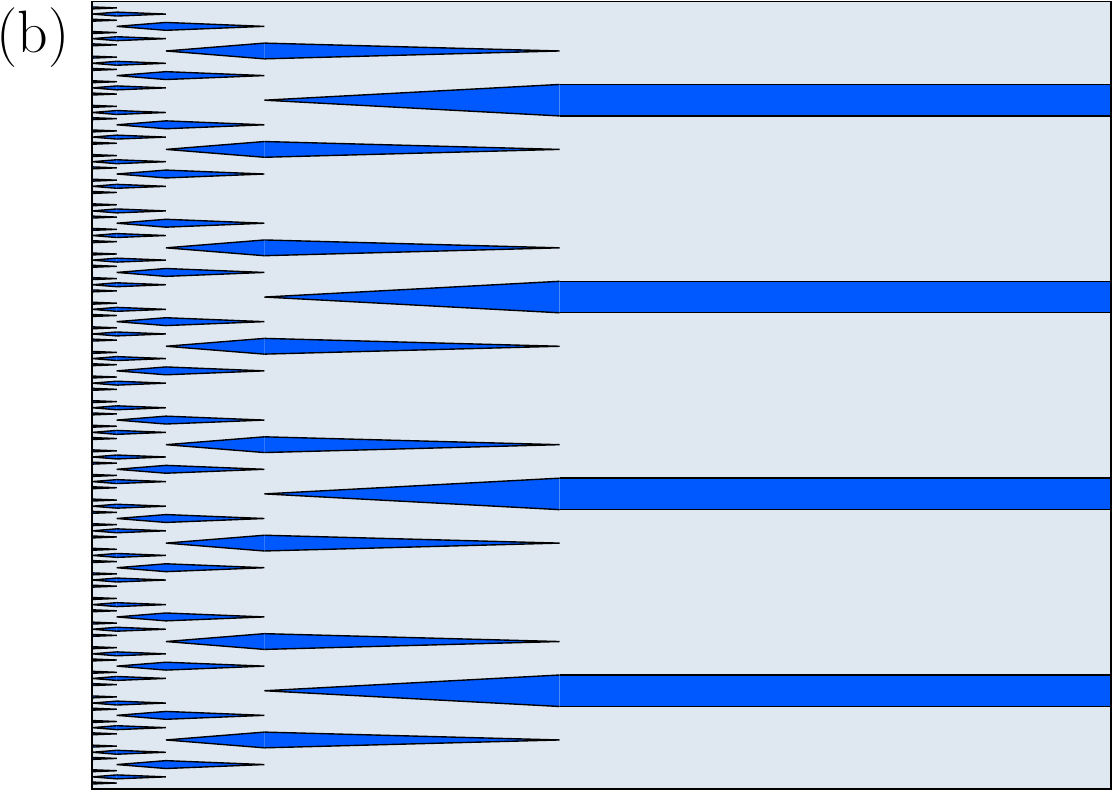}}
 \caption{(a) uniform test function $v_c=0$. (b) sketch of the branching construction $v_b$. The interface to austenite
 is at the left edge. {{Different colors indicate different values of the order parameter $\partial_2v$.}}}
\label{fig:scaling}
\end{figure}
On the one hand, the constant function $v_c:=0$ 
satisfies $I_{\theta,\eps}^p(v_c)=\theta^p$. On the other hand, if $\eps \leq \theta^p$, then a self-similarly refining function  $v_b$ as introduced in \cite[Lemma 2.1]{kohn-mueller:94} yields 
$I_{\theta,\eps}^p(v_b)\leq c\theta^p(\eps/\theta^p)^{p/(p+1)}$  (see appendix for a precise definition). For low-hysteresis shape memory alloys, one expects both parameters $\theta$ and $\eps$  to 
be small. If one of them is much smaller than the other, then one of the two mentioned regimes is 
expected to dominate the picture, as indicated by the scaling in Theorem \ref{th:scaling}. 
Theorem \ref{th:gammalimit}
addresses
the transition regime $\eps\sim\theta^p$ in the asymptotic situation $\eps$, $\theta\to 0$, and provides a step towards the better understanding of the 
onset of microstructures.

The remaining part of the paper is structured as follows. After setting some notation in {{S}}ection \ref{sec:notation}, we prove the lower bound and the accompanying compactness result, i.e., parts (i) 
and (ii) of Theorem  
\ref{th:gammalimit} in Section \ref{sec:comp}. In Section \ref{sec:ub} we prove the upper bound, i.e., part (iii) of Theorem \ref{th:gammalimit}. This is done in two steps. First,  in Subsection 
\ref{sec:recsq}, we provide explicit constructions of recovery sequences for functions whose jump sets are finite unions of segments.  Subsequently, in 
Subsection \ref{sec:density} we prove a density result which shows that it suffices to consider those simpler functions. 

\section{Notation and basic facts}\label{sec:notation}
The parameters $p\in(1,\infty)$ and $\sigma\in (0,\infty)$ are arbitrary but fixed. We write briefly
$E_\theta:=E_{\sigma,\theta}^p$ 
and $E:=E_\sigma^p$. We write $c$ and $C$ to denote generic positive constants that may change from expression to expression. We denote by $e_1$ and $e_2$ the standard basis vectors in $\R^2$. For a 
measurable set $\Omega\subset\R^n$ we denote by $|\Omega|$ its $n$-dimensional Lebesgue measure. We often do not relabel subsequences.\\
We briefly collect some definitions and properties we use in {{this work}}. For precise definitions, details and proofs we refer to \cite{AmbrosioFuscoPallara,evans-gariepy}. Let $\Omega\subset\R^2$ be 
open. 
For $u\in BV(\Omega)$, we denote the distributional gradient by $Du$, the approximate differential by $\nabla u$, the jump set by $J_u$, and the jump function by $[u]$. The jump set $J_u$ is 
countably $\calH^1$-rectifiable, and we denote the generalized normal by $\nu_u:J_u\to S^1$. We will use the decomposition $Du=D^{ac}u+D^Ju+D^Cu$, where $D^{ac}u=\nabla u\calL^2$ denotes  the absolutely continuous 
part with respect to the Lebesgue measure $\calL^2$, $D^Ju=[u]\nu_u\calH^1\LL J_u$ denotes the jump part and $D^Cu$ is the Cantor part. We denote the singular part by $D^Su:=D^Ju+D^Cu$.
The space $SBV^p(\Omega)$ contains those  functions in $BV(\Omega)$ such that the
 Cantor part vanishes, the
 {approximate differential} is in $L^p$, and the jump set has finite length.

We will use various slicing arguments (see \cite[Section 3.11]{AmbrosioFuscoPallara}). On the one hand, we use that for $u\in BV(\Omega)$, the singular parts $D^Cu$ and $D^Ju$ can be obtained 
from the respective parts of the slices. On the other hand, we use specific properties of slices of admissible functions. 
In our analysis, the space $\overline{SBV}^p_y$ as defined in \eqref{eq:admlimit} plays a crucial role. We point out that for the density part (see {{S}}ubsection \ref{sec:density}), we work with  a 
slightly different space $SBV_y^p(\Omega)$ defined on general open sets $\Omega\subset\R^2$, i.e., 
\begin{eqnarray}\label{eq:sbvpy}
SBV^p_y(\Omega):=\{u\in SBV^p(\Omega):  [u]\nu_u\in [0,\infty)e_2 \text{ }\calH^1-\text{ a.e.}\}.
\end{eqnarray}
We remark that the definition of the space $\overline{SBV_y^p}$ includes the boundary data,
whereas functions in  $SBV_y^p((0,1)^2)$ {{do not need to satisfy them}}. On the other hand, 
 $SBV_y^p$ is a subset of $SBV^p$, {{while $\overline{SBV_y^p}$ is not}} (see also Lemma \ref{lemmasbvpy} and Example \ref{example} below).

If $u\in SBV^p_y(\Omega)$, then for $\calH^1$-almost every $x_2\in\R$ the horizontal slice $u^{x_2}(t):=u(t,x_2)$ belongs to $W^{1,p}(I_{x_2})$, where
$I_{x_2}:=\{t: (t,x_2)\in\Omega\}$, and its derivative is given by $(u^{x_2})'(t)=\partial_1 u(t,x_2)$. 
Correspondingly, for $\calH^1$-almost every $x_1\in\R$ the vertical slice $u^{x_1}(s):=u(x_1,s)$ belongs to $SBV^p(I_{x_1})$, where
$I_{x_1}:=\{s: (x_1,s)\in\Omega\}$, with derivative $Du^{x_1}=\partial_2 u{\calL^1} + [u]\calH^0\LL \{s: (x_1,s)\in J_u\}$.

\section{Compactness and lower bound}\label{sec:comp}
In this section we prove parts (i) and (ii) of Theorem \ref{th:gammalimit}.
\begin{proposition}\label{prop:comp}
Suppose that $M>0$, $\theta_k\to 0$ and $u_k\in\calA$ are such that $E_{\theta_k}(u_k)\leq M$ for all $k\in\N$. Then there exist a subsequence 
and $u\in \overline{SBV}_y^p$ such that 
$u_k\to u$ strongly in $L^1((0,1)^2)$, 
$u_k\overset{\ast}{\rightharpoonup} u$ in $BV((0,1)\times(\delta,1-\delta))$ for all $\delta\in (0,\frac{1}{2})$, 
and $E(u)\leq\liminf_{k\to\infty}E_{\theta_k}(u_k)$.
\end{proposition}
\begin{proof}
 {\em Step 1: A priori bounds and identification of $u$.} We have  
  \begin{eqnarray}\label{eq:unifdx}
 \|\partial_1u_k\|_{L^1((0,1)^2)}\leq\|\partial_1u_k\|_{L^p((0,1)^2)}\leq M^{1/p},
 \end{eqnarray}
and thus, using $u_k(0,x_2)=x_2$,
 \begin{eqnarray}\label{eq:unifL1}
  \|u_k\|_{L^1((0,1)^2)}\leq \|\partial_1u_k\|_{L^1((0,1)^2)}+\|u_k(0,\cdot)\|_{L^1((0,1))} \leq M^{1/p}+\frac{1}{2}.\
   \end{eqnarray}
For $\delta\in (0,1/2)$ set
\[A_k^{\delta}:=\left\{(x_1,x_2)\in(0,1)\times(\delta,1-\delta):\ |\partial_2u_k(x_1,x_2)|\leq|\partial_2u_k(x_1,x_2)-\frac{1}{\theta_k}|  \right\} \]
and $B_k^\delta:=((0,1)\times(\delta,1-\delta))\setminus A_k^\delta$. Then, 
\begin{eqnarray}\label{eq:LpAkdelta}
\|(\partial_2u_k)\chi_{A_k^\delta}\|_{L^1((0,1)^2)}\leq\|(\partial_2u_k)\chi_{A_k^\delta}\|_{L^p((0,1)^2)}\leq M^{1/p},
\end{eqnarray}
and similarly, using $|\partial_2u_k|\leq|\partial_2u_k-\frac{1}{\theta_k}|+\frac{1}{\theta_k}$, we obtain
\begin{equation}\label{eq:unifpartialy}
 \|(\partial_2 u_k)\chi_{B_k^\delta}\|_{L^1((0,1)^2)}\leq M^{1/p}+\frac{1}{\theta_k}|B_k^\delta|.
\end{equation}
Since 
$\partial_2u_k\geq\frac{1}{2\theta_k}$ in $B_k^\delta$, we obtain
\begin{eqnarray*}                                                                                                                 
 \int_0^1\left(u_k(x_1,1-\delta)-u_k(x_1,\delta)\right)\dx=\int_{B_k^\delta}\partial_2u_k\dxy+\int_{A_k^\delta}\partial_2u_k\dxy\geq |B_k^\delta|\frac{1}{2\theta_k}-M^{1/p}.
 \end{eqnarray*}
Therefore, for almost every $x_2\in(0,\delta)$, we have
 \begin{eqnarray*}
 && \int_0^1\left(u_k(x_1,1-\delta+x_2)-u_k(x_1,x_2)\right)\dx\\
  &=&\int_0^1\left(u_k(x_1,1-\delta)-u_k(x_1,\delta)+\int_{1-\delta}^{1-\delta+x_2}\partial_2u_k(x_1,t)\dt+\int_{x_2}^\delta\partial_2u_k(x_1,t)\dt\right)\dx\\
  &\geq&|B_k^\delta|\frac{1}{2\theta_k}-M^{1/p}+\int_{(0,1)^2}\min\{\partial_2u_k(x_1,x_2),0\}\dx\geq\frac{1}{2\theta_k}|B_k^\delta|-2M^{1/p}
 \end{eqnarray*}
 where in the last step we used (\ref{eq:LpAkdelta}).
Hence,
\begin{eqnarray*}
\|u_k\|_{L^1((0,1)^2)}&\geq& \int_0^1\int_0^\delta(|u_k(x_1,1-\delta+x_2)|+|u_k(x_1,x_2)|)\dy\dx\\
&\geq&\int_0^1\int_0^\delta(u_k(x_1,1-\delta+x_2)-u_k(x_1,x_2))\dy\dx\\&\geq&\delta(\frac{1}{2\theta_k}|B_k^\delta|-2M^{1/p}),
\end{eqnarray*}
and thus, combining with \eqref{eq:LpAkdelta}, \eqref{eq:unifpartialy} and \eqref{eq:unifL1}, we obtain $\|\partial_2 u_k\|_{L^1((0,1)\times(\delta,1-\delta))}\leq\frac{C}{\delta}(M^{1/p}+1)$.
Putting things together, we have 
\begin{equation}\label{eqw11deltabound}
 \|u_k\|_{W^{1,1}((0,1)\times(\delta,1-\delta))}\leq \frac{C}{\delta}(M^{1/p}+1).
\end{equation}
Therefore, by a diagonal sequence argument, there are a subsequence 
and a limit function {$u:(0,1)^2\to\R$} such that $u_k\overset{\ast}{\rightharpoonup}u$ in $BV((0,1)\times(\delta,1-\delta))$ for all $\delta\in(0,\frac{1}{2})$. 
Further, $u_k\to u$ strongly in $L^1((0,1)\times(\delta,1-\delta))$ for all $\delta$. The same argument in (\ref{eq:unifL1})
shows that $\|u_k\|_{L^1((0,1)\times {((0,\delta)\cup (1-\delta,1))})} \le M^{1/p}\delta^{1/p'}+{2\delta}$. Therefore $u_k\to u$ strongly in $L^1((0,1)^2)$.

{\em Step 2: Show that $u\in \overline{SBV}_y^p$.} By \eqref{eq:unifdx}, there is a subsequence such that $\partial_1u_k\rightharpoonup Du\cdot e_1$ weakly in $L^p((0,1)^2)$, and in particular $Du\cdot e_1\in L^p((0,1)^2)$. Since $e_1$ is the normal to the Dirichlet boundary $\{x_1=0\}$, it follows as in the proof of Rellich's compact embedding theorem on the boundary that $u(0,x_2)=x_2$ in the sense of traces.  
Step 1 also yields $\partial_2u_k\overset{\ast}{\rightharpoonup}Du\cdot e_2$ in $\calM((0,1)\times(\delta,1-\delta))$. By \eqref{eq:LpAkdelta}, there are a subsequence and $v\in L^p((0,1)^2)$ such that $(\partial_2u_k)\chi_{A_k^\delta}\rightharpoonup v$ weakly in $L^p((0,1)^2)$. Hence, we have in $\calM((0,1)\times(\delta,1-\delta))$,
\[0\leq (\partial_2u_k)\chi_{B_\delta^k}=Du_k\cdot e_2-(\partial_2 u_k)\chi_{A_\delta^k}\overset{\ast}{\rightharpoonup} Du\cdot e_2-v=D^{ac}u\cdot e_2+D^Su\cdot e_2-v. \]
Since $D^Su\cdot e_2\perp (D^{ac}u\cdot e_2-v)$, we conclude that $D^Su\cdot e_2\geq 0$.

It remains to show that $u\in SBV_{\text{loc}}$ and that $D^{ac}u\cdot e_2-v\in L^p$. We refine the previous argument, and proceed by slicing (see \cite[Theorems 3.107 and 3.108]{AmbrosioFuscoPallara}), i.e., for almost every $x_1$, we consider the slice $u^{x_1}(\cdot):=u(x_1,\cdot)$ and show that it has locally a finite number of jumps, no Cantor part, and an absolutely continuous part which is contained in $L^p$.

By Fatou's lemma, the assumption
 $E_{\theta_k}(u_k)\leq M$  implies
that 
there is $m\in L^1((0,1))$ such that
\begin{equation}\label{eq:goodslice}\liminf_{k\to\infty}\int_0^1\min\{|(u_k^{x_1})'(x_2)|^p,|(u_k^{x_1})'(x_2)-\frac{1}{\theta_k}|^p\}\dy+\sigma\theta_k|(u_k^{x_1})''|((0,1))\leq m(x_1)
\end{equation}
for almost every $x_1$.
Analogously, by (\ref{eqw11deltabound}) {{and \eqref{eq:unifL1},}} for any $\delta$ there is $m_\delta\in L^1((0,1))$ such that
\begin{equation}\label{eq:goodslice2}
\liminf_{k\to\infty}\left(\int_\delta^{1-\delta}|(u_k^{x_1})'(x_2)|\dy+\int_0^1|u_k^{x_1}(x_2)|\dy\right)\leq m_\delta(x_1)
\end{equation}
for almost every $x_1$. 
After extracting a further subsequence $u_k^{x_1}\to u^{x_1}$ in $L^1((0,1))$
for almost every $x_1$. 
Fix one such $x_1\in(0,1)$ and $\delta\in(0,1/2)$.
By (\ref{eq:goodslice2}) the sequence $u_k^{x_1}$ is bounded in $W^{1,1}((\delta,1-\delta))$, therefore  $u_k^{x_1}\overset{\ast}{\rightharpoonup} u^{x_1}$ in $BV((\delta,1-\delta))$.

For $s\in \R$ we define
\begin{eqnarray*}
P_k^{x_1}(s):=\{x_2\in(\delta,1-\delta): (u_k^{x_1})'(x_2)>s\}.
\end{eqnarray*}
Condition (\ref{eq:goodslice}) implies continuity of $u_k^{x_1}$,
therefore for any $s\in\R$ the set
 $P_k^{x_1}(s)$ is a countable union of intervals.
By the coarea formula one has
\begin{equation*}
\sigma\theta_k
 \int_\R {{\mathcal{H}^0(}} \partial_{(\delta,1-\delta)} P_k^{x_1}(s){{)\ds}}= \sigma\theta_k|(u_k^{x_1})''|((\delta,1-\delta))\,.
\end{equation*}
Here $\partial_{(\delta,1-\delta)} P$ is the part of the measure-theoretic boundary of $P$ 
which is contained in $(\delta,1-\delta)$.
By (\ref{eq:goodslice}), for large enough $k$ the quantity on the right is smaller
than $2m(x_1)$.

We fix $\eta\in(0,1/4)$, and  choose $t_k\in {{[}}\eta/\theta_k, (1-\eta)/\theta_k{{]}}$ which minimizes
(in this interval) the quantity ${\calH^0}( \partial_{(\delta,1-\delta)} P_k^{x_1}(\cdot))$. In particular,
$P_k^{x_1}(t_k)$
is the union of at most $2m(x_1)/\sigma$ {disjoint} open intervals.
Let  $\{y_k^{(j)}:j\in\mathcal{J}\}$ be the set of midpoints
of the intervals that constitute  $P_k^{x_1}(t_k)$. Since their number is bounded, after extracting
a further subsequence we can assume $y_k^{(j)}\to z^{(j)}$, not all necessarily distinct. 
Let $K:=\{z^{(j)}: j\in\mathcal{J}\}\cap(\delta,1-\delta)$. For later reference we note that
\begin{equation}\label{eqbdjumppoints}
 2 {\calH^0(} K) \le \liminf_{k\to\infty}  {{\mathcal{H}^0(}} \partial_{(\delta,1-\delta)} P_k^{x_1}(t_k){{)}}
 \le  \liminf_{k\to\infty} \frac{\theta_k}{1-2\eta} |(u_k^{x_1})''|((\delta,1-\delta))\,.
\end{equation}
Fix now $\eps>0$, and let $I_\eps=(\delta+\eps,1-\delta-\eps)\setminus \bigcup_j (z^{(j)}-\eps,z^{(j)}+\eps)$. 
By (\ref{eq:goodslice2}) we have $|P_k^{x_1}(t_k)|\to0$. 
Since $y_k^{(j)}\to z^{(j)}$, 
for $k$ large enough we have $P_k^{x_1}(t_k)\cap I_\eps=\emptyset$. Since $\alpha\le (1-\eta)/\theta_k$ implies
$\alpha\le \eta^{-1} {|\alpha-1/\theta_k|}$ we obtain (for sufficiently large $k$)
\begin{equation*}
 \int_{I_\eps} |(u_k^{x_1})'|^p \dy\le \frac{1}{\eta^p} 
 \int_0^1\min\{|(u_k^{x_1})'(x_2)|^p,|(u_k^{x_1})'(x_2)-\frac{1}{\theta_k}|^p\}\dy\le \frac{2}{\eta^p} m(x_1).
\end{equation*}
Therefore $u_k^{x_1}\chi_{I_\eps}$ has a weak limit in $W^{1,p}(I_\eps)$, 
which coincides with $u^{x_1}$ and obeys
$\|(u^{x_1})'\|_{L^p(I_\eps)}^p \le 2\eta^{-p} m(x_1)$. Since $\eps$ was arbitrary, we conclude that
$u^{x_1}\in W^{1,p}((\delta,1-\delta)\setminus K)$;
since the set of jump points is finite, this space is a subset of $SBV^p((\delta,1-\delta))$,
 $u^{x_1}$ can only jump in the points of $K${{, and the jumps are positive}}.
Finally, since $\|(u^{x_1})'\|_{L^p{{(}}(\delta,1-\delta){{)}}}^p \le 2\eta^{-p} m(x_1)$ for all $\delta$,
we conclude $\|(u^{x_1})'\|_{L^p((0,1))}^p \le 2\eta^{-p} m(x_1)$.
Hence $u\in SBV^p_\mathrm{loc}((0,1)^2)$, with $\nabla u\in L^p((0,1)^2;\R^2)$.

{\em Step 3: Lower bound. } For the same  slices $u^{x_1}$ {{as in Step 2}}, we have 
by \eqref{eqbdjumppoints}
\[2\sigma(1-2\eta)\calH^0((\delta,1-\delta)\cap J_{u^{x_1}})\leq\liminf_{k\to\infty}\sigma\theta_k|(u_k^{x_1})''|((0,1)). \]
\sloppypar
Since $\delta$ and $\eta\in(0,1/4)$ were arbitrary, it follows that $2\sigma\calH^0(J_{{u^{x_1}}}) \leq \liminf_{k\to\infty}\sigma\theta_k|(u_k^{x_1})''|((0,1))$. Recall from Step 2 that $D^Su\cdot e_1=0$, and thus we 
have
\[\int_0^1\calH^0(J_{u^{x_1}})\dx=\calH^1(J_u). \]
Combining this with the weak $L^p$-convergences of the regular parts, we conclude by Fatou's lemma that
\begin{eqnarray*}
&&E(u)=\int_{(0,1)^2}(|\partial_1u|^p+|\partial_2u|^p)\dxy+2\sigma\calH^1(J_{u})\\
&\leq&\liminf_{k\to\infty}\int_{(0,1)^2}|\partial_1u_k|^p\dxy\\
&&+\liminf_{k\to\infty}\int_0^1\left(\int_0^1\min\{|(u_k^{x_1})'|^p,|(u_k^{x_1})'-\frac{1}{\theta_k}|^p\}\dy+\sigma\theta_k|(u_k^{x_1})''|((0,1))\right)\dx\\
&\leq&\liminf_{k\to\infty}E_{\theta_k}(u_k).
\end{eqnarray*}
This {{finishes}} the proof.
\end{proof}
Let us note that the problem does not admit global weak{{-}}$\ast$ convergence in $BV((0,1)^2)$, as the following example shows. For details and discussions we refer to \cite{diermeier:16}. 
\begin{example}\label{example}
{{Let $\alpha\in(p/(p+1),1)$, and set
\[u_k(x_1,x_2):=\begin{cases}
\frac{1}{\theta_k}x_2+\left(1-\frac{1}{\theta_k}\right)\left(1-\theta_k^\alpha x_1\right)&\text{\qquad if \ }x_2\geq 1-\theta_k^{\alpha}x_1\\
x_2&\text{\qquad otherwise}.
\end{cases} \]}}
Then $E_{\theta_k}(u_k)\leq C$ but {{$\|\partial_2u_k\|_{L^1((0,1)^2)}$ is unbounded, and hence}} there is no subsequence that converges weakly-$\ast$ in $BV((0,1)^2)$. \\
This issue {{could}} be overcome by imposing periodic boundary conditions at top and bottom of the square, i.e., $u(x_1,1)-u(x_1,0)=1$.
\end{example}

\begin{lemma}\label{lemmasbvpy}
 Let $u\in \overline{SBV}^p_y$, $\delta\in (0,1/2)$. Then 
$|Du|((0,1)\times(\delta,1-\delta))<\infty${{.}}
\end{lemma}
\begin{proof}
Let $\varphi\in C^1_c({{[0,1]\times}}(0,1);[0,1])$, with $\varphi=1$ on ${{[0,1]\times}}(\delta,1-\delta)$. 
Then
\begin{equation*}
 \int_{(0,1)^2} u \partial_2 \varphi {{\dxy}} =
 -\int_{(0,1)^2} {{(}}\partial_2 u{{)}}  \varphi {{\dxy}} 
 -\int_{(0,1)^2\cap J_u} [ u]   \varphi {{\dH}}\,.
\end{equation*}
Since $[u]\ge 0$ almost everywhere,
\begin{equation*}
 |D^Ju|((0,1)\times(\delta,1-\delta)) 
 \le \int_{(0,1)^2\cap J_u} [u]   \varphi {{\dH}}
 \le \|\partial_2 u\|_{L^1((0,1)^2)} + \|\nabla \varphi\|_{L^\infty}  \| u\|_{L^1((0,1)^2)}\,.
\end{equation*}
\end{proof}
\section{Upper bound}\label{sec:ub}
In this section, we prove the upper bound, i.e., part (iii) of Theorem \ref{th:gammalimit}. We proceed in two steps, the main difficulty being the density result in Subsection \ref{sec:density}. The 
lines of the proof are outlined in the proof of Theorem \ref{th:ub} below. 
\begin{theorem}\label{th:ub}
Let $\theta_k\to 0$ and $u\in \overline{SBV}_y^p$. Then there exist $u_k\in\calA$ such that $u_k\to u$ in
$L^1((0,1)^2)$
and 
${{E}}(u)\geq\limsup_{k\to\infty}{{E_{\theta_k}}}(u_k)$. 
\end{theorem}
{In the proof we mainly work on the space
 $SBV_y^p(\Omega)$
defined in  \eqref{eq:sbvpy}  and use on this space  the functional
\begin{equation*}
E(u,\Omega):=
  \int_{\Omega} (|\partial_1 u|^p +|\partial_2u|^p)\dxy + 2\sigma \calH^1(J_u\cap\Omega)
\end{equation*}
which reduces to $E(u)$ if $\Omega=(0,1)^2$ and $u(0,x_2)=x_2$.}
\begin{proof}
Let $u\in \overline{SBV}_y^p$, extended to $(-1,0)\times(0,1)$ by $u(x_1,x_2)=x_2$. Assume that $E(u)<\infty$, otherwise there is nothing to prove. For $\eta\in(0,1/10)$ we define
\begin{equation*}
u_\eta(x):=u(x_1-2\eta,2\eta+(1-4\eta)x_2)+2\eta (2x_2-1),  
\end{equation*}
so that $u_\eta(x)=x_2$ for $x_1\le 2\eta$, 
and let $\Omega:=(-\eta,1+\eta)\times(-\eta,1+\eta)$.
With Lemma \ref{lemmasbvpy} one obtains
$u\in SBV(({-1},1)\times{{(-\eta,1+\eta)}})$ and therefore
$u_\eta\in SBV_y^p(\Omega)$. Further, $u_\eta\to u$ in 
$L^1((0,1)^2)$ and $E(u_\eta,\Omega)\to E(u)$ as $\eta\to 0$. 
By Theorem \ref{theodensityfull} below 
(on the set $\Omega$, with $R:=\eta$ and $U:=(-\eta,\eta)\times(0,1)$) there exists a sequence ${v}_j\in SBV_y^p(\Omega)$  
with ${v}_j\to u_\eta$ in $L^1(\Omega)$, $E(v_j,\Omega)\to E(u_\eta,\Omega)$, such that 
$J_{{v}_j}$ is locally a finite union of segments, and $v_j=u_\eta$ in $U$. In particular,
${v}_j(0,x_2)=x_2$. 
Since $(0,1)^2\subset\subset\Omega$, 
$J_{v_j}\cap (0,1)^2$ is a finite union of segments.

By Lemma  \ref{lemmadensitysmooth} below, there exists a sequence $w_\ell\in SBV_y^p$ such that $w_\ell\to v_j$ in $L^1((0,1)^2)$ as $\ell\to\infty$, $w_\ell$ is smooth away from its jump set with 
smooth traces on both sides of the jumps, and the jump set consists of finitely many segments. Further, $w_\ell(0,x_2)=v_j(0,x_2)$ and $E(w_\ell)\to E(v_j)$ as $\ell\to\infty$. For every $w_\ell$, by 
Proposition \ref{prop:constr} below, there exists a recovery sequence $w_\theta^\ell\to w_\ell$ with $\limsup_{\theta\to 0}E_\theta(w_\theta^\ell)\le E(w_\ell)$. Finally, taking a diagonal sequence, 
we obtain a recovery sequence for $u$. 
\end{proof}
\subsection{Functions whose jump sets consist of finitely many segments}\label{sec:recsq}
We first provide an explicit construction of a recovery sequence for a generic function $u\in SBV_y^p((0,1)^2)$ whose jump set consists of finitely many segments.
\begin{proposition}\label{prop:constr}
 Let $u\in SBV^p_y((0,1)^2)$ with $u(0,x_2)=x_2$ be such that $J_u$ is a finite union of segments, and 
 $u$ smooth away from its jump set with 
smooth traces on both sides of the jumps. Then, for $\theta\in (0,1/2]$ there is $u_\theta\in\calA$ such that
 $u_\theta\to u$ in $L^1((0,1)^2)$ as $\theta\to 0$, $\limsup_{\theta\to 0} E_\theta(u_\theta) \le E(u)$
and $u_\theta(0,x_2)=x_2$ for all $\theta$.
 \end{proposition}
\begin{proof}
Suppose that 
  $J_u={{\bigcup}}_{k=0}^K [a_k,b_k]\times \{y_k\}$ {(up to a null set)}. Let $\rho>0$ be such that
  the distance between any pair of  segments is larger than $2\rho$.
 We shall modify $u$ inside each set $(a_k,b_k)\times (y_k-\rho, y_k+\rho)$, and not touch it outside.
 Set  $h_k(x_1):=[u](x_1,y_k)\geq 0$. 
We can assume without loss of generality that  $\theta$ is so that $\theta\|h_k\|_{L^\infty}\leq\rho$ for all $k=0,\dots, K$
 (otherwise $u_\theta(x):=x_2$ will do).
 We set
 \begin{equation*}
  u_\theta(x):=
  \begin{cases}
  u(x) + \left(\frac{x_2- y_k}{\theta} - h_k(x_1)\right) & \text{ if } y_k<x_2<y_k + \theta h_k(x_1)\,, \\
   u(x) & \text{ otherwise.}
  \end{cases}
 \end{equation*}
If $u(0,\cdot)$ is continuous, then $u_\theta(0,\cdot)=u(0,\cdot)$. 
 We set 
\[A_\theta:=\bigcup_{k=0}^KA_\theta^k,\text{\qquad }A_\theta^k:=\{(x_1,x_2):\,x_1\in (a_k,b_k),\,y_k<x_2<\theta h_k(x_1)\} \]
and observe that $|A_\theta|\le \theta \sum_k \|h_k\|_{L^\infty}\to0$ as $\theta\to0$.
Since $u=u_\theta$ outside $A_\theta$, we obtain
\begin{equation*}
 \|u-u_\theta\|_{L^1((0,1)^2)}\le |A_\theta| \max_k \|h_k\|_{L^\infty} \to0\,.
\end{equation*}
It remains to prove convergence of the energies.
Outside $A_\theta$ we have $\nabla u=\nabla u_\theta$. Inside $A_\theta$
we have $|\partial_1 u_\theta|\le \|\partial_1 u\|_{L^\infty} + \max_k \|h_k'\|_{L^\infty}$
and $|\partial_2 u_\theta-\frac1\theta|\le \|\partial_2 u\|_{L^\infty}$, therefore
\begin{equation}
 \int_{(0,1)^2} |\partial_1 u_\theta|^p+\min\{|\partial_2u_\theta|^p,\,|\partial_2u_\theta-\frac{1}{\theta}|^p\}\dxy
 \le \int_{(0,1)^2} |\partial_1 u|^p+|\partial_2u|^p\dxy + c |A_\theta|\,.
 \label{eq:simpleelast}
\end{equation}
We next consider the surface energy term. By construction, $Du_\theta\in SBV((0,1)^2)$ with  $J_{Du_\theta}\subset\partial A_\theta$,
so that
$\calH^1(J_{Du_\theta})\le \sum_k \int_{a_k}^{b_k} {{(}}1 + \sqrt{1+\theta^2(h_k')^2(t)}{{) \dt}}
\to 2 \sum_k (b_k-a_k) = 2\calH^1(J_u)$. At the same time, $|[Du_\theta]|\le \theta^{-1} + \max_k \|h_k'\|_{L^\infty}$.
Further, $\nabla^2 u_\theta=\nabla^2 u+\nabla^2h_k(x_1)\chi_{A_\theta^k}$. Since $\| h_k''\|_{L^\infty}\leq c$ and $|A_\theta^k|\leq c\theta$, we estimate
\begin{eqnarray}\label{eq:simplesurf}
\sigma\theta|D^2u_\theta|((0,1)^2)&\leq&\sigma\theta\left(\|\nabla^2 u\|_{L^1((0,1)^2)}+c|A_\theta|+(\frac{1}{\theta}+c)\calH^1(J_{Du_\theta})\right)\nonumber\\
&\leq&2\sigma(1+c\theta)\calH^1(J_u)+c\sigma\theta.
\end{eqnarray}
Putting together \eqref{eq:simpleelast} and \eqref{eq:simplesurf}, we obtain $\limsup_{\theta\to 0}E_\theta(u_\theta)\leq E(u)$. This concludes the proof.
\end{proof}

\begin{lemma}\label{lemmadensitysmooth}
 Let  $u\in SBV^p_y((0,1)^2)$ be such that $J_u$ is a finite union of segments and such that $u(0,x_2)=x_2$. Then there is  a sequence 
 $v_j\in  SBV^p_y((0,1)^2)$ with the following properties: We have $v_j\to u$ in $L^1$, each $v_j$ is smooth (up to the boundary) away from the jump set, which consists of finitely many segments, the traces
 on both sides of the jump are smooth and $v_j(0,x_2)=x_2$. Further, 
 $\limsup_{j\to\infty} E(v_j) \le E(u)$.
\end{lemma}
\begin{proof}
 For $\eps>0$ let $\varphi_\eps\in C^\infty_c((-\eps,\eps))$ be {{a}} one-dimensional mollifier. 
 Let $J_u={{\bigcup}}_{k=0}^K [a_k,b_k]\times \{y_k\}$ (up to null sets). 
 
 We first mollify in the horizontal direction. Extend $u$ by $u(x_1,x_2)=x_2$ for $x_1\leq 0$.
 We set, for $x\in (-\eps, 1-\eps)\times(0,1)$,
 \begin{equation*}
  w_\eps(x_1,x_2){{:=}}\int_\R \varphi_\eps(x_1-t) u(t,x_2) \dt \,.
 \end{equation*}
  We remark that $w_\eps\in SBV$ and its jump set is contained in the segments which arise from those where $u$ jumps
  by making them $\eps$-longer on each side. 
  Precisely, $J_{w_\eps}\subset {{\bigcup}}_{k=0}^K [a_k-\eps,b_k+\eps]\times \{y_k\}$ (up to null sets), hence its length
  has grown by at most $2{(K+1)}\eps$.
  By convexity $\|\partial_i w_\eps\|_{L^p{{(}}(-\eps,1-\eps)\times(0,1){{)}}}\le \|\partial_i u\|_{L^p((0,1)^2)}$ for $i=1,2$.
  Further, if $[u]\ge0$ then $[w_\eps]\ge 0$, since the new jump arises from the old one by averaging. 
  
  At this point we mollify in the vertical direction, separately in each region without discontinuities. 
  In order to keep the boundary data it is convenient to subtract
  the affine function $x\mapsto x_2$.
  Let $a,b\in (0,1)$ be two consecutive elements of $\{0,1\}\cup \{y_k\}_{k=0,\dots, K}$, i.e., two distinct values such that 
  $(0,1)\times(a,b)\cap J_u$ is an $\calH^1$-null set. 
 We focus on the construction for $x_2\in (a,b)$, the other regions are analogous. We can assume $b> a+2\eps$.
 We set
 \begin{equation*}
  \hat z_\eps(x_1,x_2){{:=}}\begin{cases}         w_\eps(x_1, a)-a&\text{ if } x_2\le a\\
w_\eps(x_1,x_2)-x_2 &\text{ if } a<x_2<b\\
       w_\eps(x_1, b) -b&\text{ if } x_2\ge b
       \end{cases}
 \end{equation*}
and $z_{\tilde \eps}$ as the vertical mollification of $\hat z_\eps$ on a scale $\tilde\eps\leq \eps$, i.e. $z_{\tilde \eps}(x_1,x_2){{:=}}\int_\R \varphi_{\tilde \eps}(x_2-t) \hat z_\eps(x_1,t) \dt$. 
We remark that $z_{\tilde\eps}(x_1,x_2)=w_\eps(x_1,b)-b$ for all $x_2\ge b+\tilde\eps$.

Finally, we scale back so that  $(-\eps,1-\eps)\times (a-\tilde\eps, b+\tilde\eps)$ is mapped to $(0,1)\times (a,b)$.
  Precisely, we define
  \begin{equation}
   \hat v_{\tilde\eps}(x_1,x_2) {{:=}} z_{\tilde \eps}\left( x_1-\eps, a-\tilde\eps + (x_2-a) \frac{b+2\tilde\eps-a}{b-a} \right) {+x_2}.
  \end{equation}
This gives a smooth map on the set $[0,1]\times [a,b]$, because the construction is the same
as mollification (of an extension) with $\psi_{\eps,\tilde \eps}(x_1,x_2)=\varphi_\eps(x_1)\varphi_{\tilde\eps}(x_2)\in C^\infty_c((-\eps,\eps)\times (-\tilde \eps,\tilde \eps))$, {{combined with}} an affine rescaling.
The only term in the derivative that is not automatically estimated by convexity is the $x_1$-derivative close to the jumps. However{{,}} $\| \partial_{1} \hat 
v_{\tilde\eps}\|^p_{L^p((0,1)\times(b-\tilde\eps,b))}\leq \| \partial_{1} w_\eps \|^p_{L^p((0,1)\times(b-\tilde\eps,b))} +\tilde\eps\|\partial_{1} w_\eps(b,\cdot)\|^p_{L^p((0,1))}$ where the second 
term gets small for small $\tilde \eps$.
  The same is done in all other intervals.
On the lines of the jumps, $\hat v_{\tilde\eps}$ coincides with $w_\eps$, which has a jump of the controlled length. Further, the jump has the appropriate sign and is smooth.
Close to $a$ and $b$ the function $v_\eps(x){-x_2}$ does not depend on $x_2$, therefore the  constructions on the two sides match smoothly away from the jump.
\end{proof}

\subsection{Density}\label{sec:density}
In this section we prove density of functions with regular jump sets.
{The argument is more naturally discussed in general domains, but with 
a jump set of finite total length, therefore we} use the 
space $SBV^p_y$ as introduced in \eqref{eq:sbvpy}. The density proof is done by three different constructions in 
small squares, that we present first, followed by a somewhat involved covering argument, discussed later. The constructions
in turn build upon suitable variants of the Poincar\'e inequality.

We write $Q_\rho:=(-\rho,\rho)^2$, $Q_\rho(x):=x+Q_\rho$. 
For two functions $u,v:\Omega\to\R^m$ we say that $u=v$ around $\partial\Omega$ if there is $\omega\subset\subset\Omega$ such that
$u=v$ in $\Omega\setminus\omega$.
For the entire construction we fix 
a mollifier $\varphi_\rho\in C^\infty_c(Q_{\rho/2};[0,\infty))$ with $|D\varphi_\rho|\le {{c/\rho^3}}$.

\subsubsection{Constructions: Type I (bulk squares)}
We first treat squares which contain a small amount of jump set. Since the jump set is purely horizontal, one can use the normal Poincar\'e inequality
to control the dependence of $u$ on $x_1$ (Lemma \ref{lemmapoincarehoriz}). At the same time, if the length of the jump set is not enough to divide the square into two halves,
there are some vertical sections on which one can also use Poincar\'e. This gives a control of the $L^p$ oscillation of $u$ in terms of the $L^p$ norm
of its regular gradient $\nabla u$ (Lemma \ref{lemmatyp1mollifsquare}), and therefore permits to estimate the gradient of a mollification
(Lemma \ref{lemmatyp1enmoll}). An interpolation around the boundary leads to a construction which makes $u$ smooth in the interior of the square
(Proposition \ref{proptyp1}), which is the main result of this subsection.

\begin{lemma}[Horizontal Poincar\'e]\label{lemmapoincarehoriz}
 If $u,v\in SBV^p_y(Q_r(x_*))$ and $u=v$ on $\partial Q_r(x_*)$, then
 \begin{equation*}
  \|u-v\|_{L^p(Q_r(x_*))} \le c r \|\nabla u-\nabla v \|_{L^p(Q_r(x_*))}\,.
 \end{equation*}
The same holds if  $u=v$ on $\partial Q_r(x_*)$ is replaced by the fact that 
there is $t\in [-r,r]$ such that $u=v$ on $(x_*)+\{t\}\times(-r,r)$ in the sense of traces.
 \end{lemma}
\begin{proof}
 It suffices to apply the one-dimensional Poincar\'e estimate in the $x_1$-direction.
\end{proof}

\begin{lemma}[Poincar\'e with discontinuities]\label{lemmatyp1mollifsquare}
There is $c>0$ such that for any {{$\rho>0$ and any}}
  $u\in SBV^p_y(Q_\rho(x_*))$ with $\calH^1 (J_u\cap Q_\rho(x_*))\le \rho$ there is $\bar u\in\R$ such that
  \begin{equation}
  \|u-\bar u\|_{L^p(Q_\rho(x_*))} \le c \rho \|\nabla u\|_{L^p(Q_\rho(x_*))}\,.
 \end{equation}
 \end{lemma}
\begin{proof}
Without loss of generality $x_*=0$.
Since horizontal slices of $u$ are in $W^{1,p}$, 
by the one-dimensional Poincar\'e inequality
one has
\begin{equation}\label{eqpoinchoriz}
\int_{-\rho}^\rho  |u(x_1,x_2)-u(x_1',x_2)|^p \dy \le (2\rho)^{p-1} \int_{Q_\rho} |\partial_1u|^p \dxy
\end{equation}
 for almost every $x_1$, $x_1'\in (-\rho,\rho)$.
For fixed $x_1$, consider now the vertical slices $t\mapsto u^{x_1}(t):=u(x_1,t)$. By the $SBV$ slicing theorem, 
since $\calH^1(J_u)\le\rho$,
there is a set $E\subset(-\rho,\rho)$ of measure at least $\rho$ such that $u^{x_1}\in W^{1,p}((-\rho,\rho))$ for $x_1\in E$. 
Further, by Fubini's theorem there is a subset $E'\subset E$ of measure at least $\rho/2$ such that $\int_{-\rho}^\rho{{ |\partial_2 u(x_1',t)|^p \dt}} \le 2/\rho \int_{Q_\rho} |\partial_2u|^p {{\dxy}}$ {{ for all $x_1'\in E'$}}.
We choose $x_1'\in E'$ such that (\ref{eqpoinchoriz}) holds for almost every $x_1\in(-\rho,\rho)$. By the one-dimensional Poincar\'e inequality there is $\bar u\in\R$ such that
\begin{equation*}
\int_{-\rho}^\rho  |u(x_1',x_2)-\bar u|^p \dy\le (2\rho)^p
\int_{-\rho}^\rho  |\partial_2 u(x_1',x_2)|^p \dy\le 2^{p+1} \rho^{p-1}
\int_{Q_\rho} |\partial_2u|^p \dxy\,.
\end{equation*}
Combining with (\ref{eqpoinchoriz}) and integrating in $x_1$ we conclude
\begin{equation*}
\int_{Q_\rho}  |u(x_1,x_2)-\bar u|^p \dy\le c \rho^p\int_{Q_\rho} |\nabla u|^p \dxy\,.
\end{equation*}
This finishes the proof.
\end{proof}

\begin{lemma}[Mollification with small jump set]\label{lemmatyp1enmoll}
 Let $u\in SBV^p_y(Q_r(x_*))$, with $\calH^1(J_u\cap Q_r(x_*))\le \rho$, { for} $0<\rho<r$. Then
 \begin{equation*}
  \int_{Q_{r-\rho}(x_*)} |\nabla (u\ast \varphi_\rho)|^p \dxy \le 
\left(1+c \frac{\calH^1(J_u\cap Q_r(x_*))^{1/p'}}{\rho^{1/p'}}\right)^p\int_{Q_r(x_*)} |\nabla u|^p {{\dxy}} \,.
  \end{equation*}
\end{lemma}
\begin{proof}
We can assume $x_*=0$.
 Fix $x\in Q_{r-\rho}$. Since horizontal slices of $u$ are in $W^{1,p}$, we have
 \begin{align}
\partial_1 (u\ast \varphi_\rho)(x)
&= \int_{Q_\rho(x)} \partial_1 u(y)  \varphi_\rho(x-y)\dyy=((\partial_1 u)\ast \varphi_\rho)(x).
\label{eqpartintx1}
 \end{align}
 The other component is more subtle, since $u$ jumps in the $x_2$ direction. 
 Let $\omega_x:=\{y\in Q_\rho(x): y+\R e_2\cap J_u\cap Q_\rho(x) \ne\emptyset\}$. We write
 \begin{equation*}
\partial_2 (u\ast \varphi_\rho)(x)= \int_{Q_\rho(x)} u(y) \partial_2 \varphi_\rho(x-y)\dyy
 \end{equation*}
and separate the integral into the part in $\omega_x$ and the part outside it. Outside $\omega_x$ we can integrate
by parts in the ${ x_2}$ direction. Therefore for any $\bar u\in\R$
 \begin{equation*}
  \partial_2 (u\ast \varphi_\rho)(x) = \int_{Q_\rho(x)\setminus\omega_x} \partial_2 u(y)\varphi_\rho(x-y) \dyy
  + \int_{\omega_x} \partial_2\varphi_\rho (x-y) (u(y)-\bar u) \dyy.
 \end{equation*}
 Let  $R(x):= \int_{\omega_x} \partial_2\varphi_\rho (x-y) (u(y)-\bar u) \dyy$ denote the last term. 
Combining with (\ref{eqpartintx1}) gives
 \begin{equation*}
\nabla (u\ast \varphi_\rho)(x) = \int_{Q_\rho(x)\setminus\omega_x} \nabla u(y)\varphi_\rho(x-y) \dyy
  + \int_{\omega_x} (\partial_1 u,0)(y)\varphi_\rho(x-y) \dyy + R(x)e_2
 \end{equation*}
 which immediately gives
 \begin{equation}\label{equastur}
|\nabla (u\ast \varphi_\rho)|(x) \le (|\nabla u|\ast \varphi_\rho)(x)+|R|(x).
 \end{equation}
 It remains to estimate $R$. Since $|\nabla\varphi_\rho|\le c/\rho^3$, H\"older's inequality yields
 \begin{equation*}
  |R|(x)\le \frac{c}{\rho^3}\|u-\bar u\|_{L^p(Q_\rho(x))} |\omega_x|^{1/p'}
 \end{equation*}
 where as usual $p'=p/(p-1)$.
 By Lemma \ref{lemmatyp1mollifsquare}, with the appropriate choice of $\bar u$ we have $\|u-\bar u\|_{L^p(Q_\rho(x))}\le c \rho \|\nabla u\|_{L^p(Q_\rho(x))}$.
We observe that  
$|\omega_x|\le 2\rho \calH^1(J_u\cap Q_\rho(x))$ and
$ \|\nabla u\|_{L^p(Q_\rho(x))}^p=(\chi_{Q_\rho}\ast |\nabla u|^p)(x)$, where $Q_\rho=Q_\rho(0)$. Therefore
 \begin{equation*}
  |R|^p(x)\le \frac{c \calH^1(J_u\cap Q_r)^{p/p'}}{\rho^{2p-p/p'}} (\chi_{Q_\rho}\ast |\nabla u|^p)(x)
 \end{equation*}
 for all $x\in Q_{r-\rho}$.
Integrating over $x$ leads to
 \begin{equation*}
  \|R\|_{L^p(Q_{r-\rho})}^p \le   \frac{c \calH^1(J_u\cap Q_r)^{p/p'}}{\rho^{2p-p/p'}} \rho^2\|\nabla u\|_{L^p(Q_r)}^p
 \end{equation*}
and, recalling (\ref{equastur}) and using $2-1/p'-2/p=1/p'$, we conclude
\begin{equation*}
 \|\nabla (u\ast \varphi_\rho)\|_{L^p(Q_{r-\rho})} \le\|\nabla u\|_{L^p(Q_r)}
 +
\frac{c \calH^1(J_u\cap Q_r)^{1/p'}}{\rho^{1/p'}} \|\nabla u\|_{L^p(Q_r)}\,.
\end{equation*}
\end{proof}

We are now in the position to state the main result of this subsection. 

\begin{proposition}\label{proptyp1}
 Let $u\in SBV^p_y(Q_r)$ with $\calH^1(J_u\cap Q_r)\le \rho\le r/10$. Then there is $v\in SBV^p_y(Q_r)$
 such that $v\in W^{1,p}(Q_{r/2})$, $v=u$ around $\partial Q_r$, $\calH^1(J_v\setminus J_u)=0$, 
 \begin{equation*}
  \|\nabla v\|_{L^p(Q_r)} \le \left(1+c \left(\frac{\calH^1(J_u\cap Q_r)}{\rho}\right)^{1/p'} + c\frac\rho r\right) \|\nabla u\|_{L^p(Q_r)}\,,
 \end{equation*}
and $|Dv|(Q_r)\le c|Du|(Q_r)$.
\end{proposition}
\begin{proof}
 We select $R\in (r/2, r-3\rho)$ such that
\begin{equation*}
 \|\nabla u\|_{L^p(Q_{R+3\rho}\setminus Q_{R-2\rho})}^p \le \frac{c \rho}{r} \|\nabla u\|^p_{L^p(Q_r)}\,.
\end{equation*}
Fix $\psi\in C^\infty_c(Q_{R+\rho};[0,1])$ with $\psi=1$ in $Q_{R}$ and $|\nabla\psi|\le c/\rho$.
We set
\begin{equation*}
 v := \psi (u\ast\varphi_\rho) + (1-\psi) u\,,
\end{equation*}
so that $J_v\subset J_u\setminus Q_R$ 
and $[v]=(1-\psi)[u]$ on $J_v$ (both up to $\calH^1$-null sets), which implies $v\in SBV^p_y(Q_r)$.
Inside $Q_{R}$ we have $v=u\ast\varphi_\rho\in C^\infty$, outside $Q_{R+\rho}$ we have $v=u$. 
By Lemma \ref{lemmatyp1enmoll} applied to $Q_{R+\rho}$ we have
$\|\nabla v\|_{L^p(Q_R)} \le \|\nabla u\|_{L^p(Q_{R+\rho})} (1+ c (\calH^1(J_u)/\rho)^{1/p'})$.

The {{elastic}} energy in the interpolation region
is controlled by 
\begin{align*}
\|\nabla v\|_{L^p(Q_{R+\rho}\setminus Q_R)}^p
\le& c\|\nabla (u\ast\varphi_\rho)\|_{L^p(Q_{R+\rho}\setminus Q_R)}^p + c \|\nabla u\|_{L^p(Q_{R+\rho}\setminus Q_R)}^p \\
&+ \frac{c}{\rho^p}\| (u\ast\varphi_\rho)-u\|_{L^p(Q_{R+\rho}\setminus Q_R)}^p.
\end{align*}
We can cover this set by $c\lfloor r/\rho\rfloor$ squares $Q_\rho(y_i)$ of side length $2\rho$,
such that the larger squares $Q_{2\rho}(y_i)$ have finite overlap and are contained in 
$Q_{R+3\rho}\setminus Q_{R-2\rho}$.
 By Lemma \ref{lemmatyp1enmoll} with $r=2\rho$ we have, for each $i$,
$\|\nabla (u\ast\varphi_\rho)\|_{L^p(Q_\rho(y_i))}\le c 
\|\nabla u\|_{L^p(Q_{2\rho}(y_i))}$.
By  Lemma \ref{lemmatyp1mollifsquare} applied to $Q_{2\rho}(y_i)$ there is $\bar u\in\R$ with
$\| u-\bar u\|_{L^p(Q_{2\rho}(y_i))}\le c\rho \|\nabla u\|_{L^p(Q_{2\rho}(y_i))}$.
This implies
$\| \varphi_\rho\ast u-\bar u\|_{L^p(Q_{\rho}(y_i))}\le c\rho \|\nabla u\|_{L^p(Q_{2\rho}(y_i))}$, and therefore
$\| \varphi_\rho\ast u- u\|_{L^p(Q_{\rho}(y_i))}\le c\rho \|\nabla u\|_{L^p(Q_{2\rho}(y_i))}$.
Summing over all squares we obtain
\begin{equation*}
\|\nabla v\|_{L^p(Q_{R+\rho}\setminus Q_R)}^p
\le  c \|\nabla u\|_{L^p(Q_{R+3\rho}\setminus Q_{R-2\rho})}^p\le \frac{c \rho}{r} \|\nabla u\|^p_{L^p(Q_r)}\,.
\end{equation*}
Collecting terms, we have
\begin{equation*}
\|\nabla v\|_{L^p(Q_r)}^p \le
 (1+ c (\calH^1(J_u)/\rho)^{1/p'})^p\|\nabla u\|_{L^p(Q_{R+\rho})}^p
 +  \frac{c \rho}{r} \|\nabla u\|^p_{L^p(Q_r)} + \|\nabla u\|_{L^p(Q_r\setminus Q_{R+\rho})}^p
\end{equation*}
which concludes the proof of the first estimate. To control the total variation of $Dv$, by convexity
one only needs to estimate the term $\|(u-u\ast \varphi_\rho)D\psi\|_{L^1(Q_r)}$, which is done 
using the Poincar\'e-type estimate $\|u-u\ast{{\varphi_\rho}}\|_{L^1(Q_{r-\rho})} \le c \rho |Du|(Q_r)$ and $|D\psi|\le c/\rho$.
\end{proof}

\subsubsection{Constructions: Type II (bad squares)}
We consider here squares where a substantial amount of jump set is present, without making any other structural assumption. 
The {{part}} where {{the jump set}} is approximately
flat will be treated later, so that this construction will be used only on ``bad'' squares, where the jump set is present but irregular (on the scale
of the square).
\begin{proposition}\label{proptyp2}
 For any $u\in SBV_y^p(Q_r)$ there is $v\in SBV_y^p(Q_r)$ with $u=v$ around $\partial Q_r$, 
 $\|\nabla v\|_{L^p(Q_r)}\le c \|\nabla u\|_{L^p(Q_r)}$, $\calH^1(J_v)\le c \calH^1(J_u)$, $|Dv|(Q_r)\le c |Du|(Q_r)$,
 and such that $J_v\subset S\cup (J_u\setminus Q_{r/2})\cup N$, with $S$ the union of at most $c\calH^1(J_u\cap Q_r)/r$ horizontal segments and $\calH^1(N)=0$.
\end{proposition}
\begin{proof}
 Pick $s\in (- r, r)$ such that {$[u](s,x_2)\ge 0$ for all $x_2$,}
 \begin{align*}
 \sum_{J_u\cap (se_1+\R e_2)} [u] \le \frac c r \int_{J_u} [u] \dH,
 && \calH^0\left(J_u\cap (se_1+\R e_2)\right)\le \frac{c}{r} \calH^1(J_u\cap Q_r),\\
   \text{and } & & \int_{-r}^r |\partial_2 u|^p(s,x_2) \dy \le \frac{c}{r} \int_{Q_r} |\partial_2 u|^p \dxy,
 \end{align*}
 in the sense of slicing.
 Set $w(x_1,x_2):=u(s,x_2)$. Then $w\in SBV^p(Q_r)$, $S:=J_w$ is the union of at most $ \frac{c}{r} \calH^1(J_u)$ horizontal segments
 of length $2r$ and $\|\nabla w\|_{L^p(Q_r)}\le c \|\nabla u\|_{L^p(Q_r)}$. Moreover,
 \begin{eqnarray}\label{eq:A}
{|D^Jw|(Q_r)}=\int_{J_w} [w] \dH =2r  \sum_{J_u\cap (se_1+\R e_2)} [u] \le c \int_{J_u} [u] \dH\,.
 \end{eqnarray}
 It remains to interpolate. We fix $\psi\in C^\infty_c(Q_{r};[0,1])$ with $\psi=1$ on $Q_{r/2}$ and $|\nabla\psi|\le c/r$, and
 set
\begin{equation*}
 v := \psi w + (1-\psi) u\,,
\end{equation*}
so that $J_v\subset (J_u\setminus Q_{r/2})\cup J_w\cup N$ with $\calH^1(N)=0$. 
Further, from $[v]=\psi[w]+(1-\psi)[u]$ and $0\le\psi\le1$ one easily verifies that $[u]\ge0$ implies $[v]\ge0$,
hence $v\in SBV^p_y(Q_r)$ {{with $\calH^1(J_v)\leq c\calH^1(J_u)$}}.

 By the Poincar\'e inequality in the $x_1$ direction (Lemma \ref{lemmapoincarehoriz}), $w=u$ on $\{x_1=s\}$ and the choice of $s$, we have
 \begin{equation*}
  \|w-u\|_{L^p(Q_r)}\le cr \|\nabla u\|_{L^p(Q_r)}\,.
 \end{equation*}
 Therefore
 \begin{equation*}
  \|\nabla {{v}}\|_{L^p(Q_r)}\le 
 c\|\nabla u\|_{L^p(Q_r)} + \frac{c}{r} \|w-u\|_{L^p(Q_r)}
 + \|\nabla w\|_{L^p(Q_r)}\le c \|\nabla u\|_{L^p(Q_r)} \,.
 \end{equation*}
The {{estimate for $|Dv|$ is}} done as in Proposition \ref{proptyp1}{{, using \eqref{eq:A}}}.
 \end{proof}

\subsubsection{Constructions: Type III (jump squares)}
{{We finally deal with squares that contain regular parts}} of $J_u$, 
in the sense of blow-ups. Although the normal is almost everywhere $e_2$, the main part of the jump {{set}} needs not be a segment,
{even locally,} but can
be a countable union of segments contained in a $C^1$ curve. 

\begin{proposition}\label{proptyp3}
 Let $u\in SBV^p_y(Q_r)$ and $\rho\in (0,r/4)$ 
  be such that there is $\gamma\in C^1((-r,r);{{(-r/2,r/2)}})$ with
 $\calH^1(J_u \Delta \{(x_1,\gamma(x_1){{)}}: x_1\in (-r,r)\})<\rho/4$.
 
 Then there is $v\in SBV^p_y(Q_r)$ such that $u=v$ around $\partial Q_r$, 
 {$J_v\subset (J_u\setminus Q_{r-\rho}) \cup S\cup N$, with $S$ the union 
 of two segments and  $\calH^1(N)=0$,}
 $\|\nabla v\|_{L^p(Q_r)}\le c (r/\rho)\|\nabla u\|_{L^p(Q_r)}$, 
$\calH^1(J_v\cap Q_{r-\rho})\le 2r+2\rho$, and 
$\calH^1(J_v\setminus Q_{r-\rho})\le c\rho $.
\end{proposition}
\begin{proof}
For $x_+\in (r-\rho,r)$ we consider the slice $u^{x_+}(\cdot):=u(x_+, \cdot)$.
By assumption, the set of $x_+$ such that the jump set of $u^{x_+}$ does not coincide with {$\{\gamma(x_+)\}$} has measure  no larger than $\rho/4$. 
By Fubini, the set of $x_+$ such that $\|(u^{x_+})'\|_{L^p((-r,r))}^p\ge \frac{2}\rho \|\nabla u\|_{L^p(Q_r)}^p$
has measure no larger than $\rho/2$.
Further, for almost all $x_+$, $[u^{x_+}]\ge 0$ on its jump set.
The same holds on the other side. Therefore we can choose
$x_-\in (-r,-r+\rho)$ and $x_+\in (r-\rho,r)$ such that
\begin{equation*}
 J_{u^{x_\pm}}=\{y_\pm\},\hskip3mm [u^{x_\pm}](y_\pm)\ge 0,\hskip3mm \text{ and } \hskip3mm \|(u^{x_\pm})'\|_{L^p((-r,r))}^p\le \frac{2}\rho \|\nabla u\|_{L^p(Q_r)}^p\,,
\end{equation*}
where  $y_\pm:=\gamma(x_\pm)\in (-r/2,r/2)$.

Let $\delta\in (0,r/2)$, chosen below.
We define $w:Q_r\to\R$ as $u$ outside $(x_-,x_+)\times(-r,r)$, as the value at $x_1=x_\pm$ in $((x_-,x_+)\setminus (-\delta,\delta))\times (-r,r)$, and as
the linear interpolation inside. Precisely,
\begin{equation*}
 w(x_1,x_2):=
 \begin{cases}
 u(x_1,x_2) & \text{ if } -r\le x_1<x_-\\
  u(x_-,x_2) & \text{ if } x_-\le x_1\le -\delta\\[1mm]
  \displaystyle
  \frac{x_1+\delta}{2\delta} u(x_+,x_2)+
  \frac{\delta-x_1}{2\delta} u(x_-,x_2) & \text{ if } -\delta< x_1< \delta\\[1mm]
  u(x_+,x_2) & \text{ if } \delta\le x_1\le x_+\\
 u(x_1,x_2) & \text{ if } x_+\le x_1<r\,.
 \end{cases}
\end{equation*}
We observe that $J_w\cap {{(}}(x_-,x_+)\times(-r,r){{)}}$ is the union of two segments of length $x_++\delta$ and $\delta-x_-$,
located at $x_2=y_\pm$. Further,
the jump of $w$ is a convex combination of the jumps of $u^{x_\pm}$, therefore nonnegative,
so that $w\in SBV^p_y(Q_r)$.
We now estimate the derivatives. By convexity, $|\partial_2 w|(x_1,x_2)\le |\partial_2 u|(x_-,x_2)+
|\partial_2 u|(x_+,x_2)$ for all $x_1\in (x_-,x_+)$. The horizontal derivative  vanishes outside $(-\delta,\delta)$, and obeys
\begin{equation*}
|\partial_1w|(x_1,x_2)=\frac{|u(x_+,x_2)-u(x_-,x_2)|}{2\delta}
\le\frac{1}{2\delta} \int_{x_-}^{x_+}|\partial_1 u| (x_1,x_2)\dx
\end{equation*}
inside,
which implies, using first H\"older's inequality and then integrating,
\begin{equation*}
 \|\nabla w\|_{L^p((x_-,x_+)\times(-r,r))}^p \le c\left(\left(\frac{r}{\delta}\right)^p+\frac{r}{\rho}\right)
  \|\nabla u\|_{L^p(Q_r)}^p \,.
\end{equation*}
Further, by Lemma \ref{lemmapoincarehoriz}
\begin{equation*}
 \|u-w\|_{L^p(Q_r)} \le c r \|\nabla u-\nabla w\|_{L^p(Q_r)}\,.
\end{equation*}
Let now $\psi\in C^\infty_c((-r,r))$ with $\psi=1$ on $(-r+\rho,r-\rho)$ and $|\psi'|\le c/\rho$, and define
\begin{equation*}
 v(x):=\psi(x_2) w(x) + (1-\psi(x_2)) u(x)\,. 
\end{equation*}
Obviously $v\in SBV^p_y(Q_r)$, and the jump has the stated properties.
In particular, $\calH^1(J_v\setminus Q_{r-\rho})\le 3\rho $
because $\gamma'=0$ $\calH^1$-almost everywhere on the set $\{x_1: (x_1,\gamma(x_1))\in J_u)\}$, 
and  $\calH^1(J_u \Delta \{(x_1,\gamma(x_1){{)}}: x_1\in (-r,r)\})<\rho/4$.
The energy of the interpolation is controlled by
 \begin{align*}
  \|\nabla v\|_{L^p(Q_r)} \le c\left(\frac{r}{\rho}+\frac{r}{\delta}+\left(\frac{r}{\rho}\right)^{1/p}\right)\|\nabla u\|_{L^p(Q_r)}{{.}}
 \end{align*}
We finally choose $\delta:=\rho$ and conclude the proof.
 \end{proof}

\subsubsection{Covering and global approximation}
We start with a covering Lemma. We need to cover an open set $\Omega$ by a family of squares which have finite overlap, 
such that the half-as-large squares still cover $\Omega$, and such that the ``overlap chains'' are bounded{{.}}
For this purpose we define the set of $k$-neighbouring squares $\calN_k(q)$ and state some of its properties. For a square $q\subset\R^2$ we denote by $\ell_q$ its {half side length,
so that $q=Q_{\ell_q}(x)$ for some $x$}.
\begin{lemma}[Covering]\label{lemmacover}
Let $\Omega\subset\R^2$ open,  $\delta>0$. Then there are
 $N$ families of squares $\calF_1$, \dots, $\calF_N$, all contained in $\Omega$ and with side length no larger than $\delta$, such that,
with $\calF:=\cup_k \calF_k$
and  $\hat q$ denoting the square with the same center and half the side length as $q$,
\begin{enumerate}
  \item $\Omega= \bigcup_{q\in \calF} \hat q$;
   \item if $q\cap q'\ne \emptyset$, then $\frac1c\ell_q\le\ell_{q'}\le c\ell_q$, for all $q,q'\in\calF$;
  \item for each $k$ the squares in $\calF_k$ are disjoint.
 \end{enumerate} 
Further, for $q\in\calF$ let $\calN_1(q):=\{q'\in\calF: q\cap q'\ne\emptyset\}$ be the set of its first neighbours, and $\calN_{k+1}(q):=\cup_{q'\in\calN_k(q)} \calN_1(q')$
be the set of $k$-neighbours.
Then $\# \calN_k(q)\le b^k$ and $\dist(q,q')+\ell_{q'}\le a^k \ell_q$ 
for all $q'\in \calN_k(q)$. 

The constants $N$, $a$, $b$ and $c$ are universal. 
 \end{lemma}
\begin{proof}
This is a variant of Whitney's covering argument, similar to the one used for example in \cite[Theorem~3.1]{FrieseckeJamesMueller2002}.
For a proof we refer to the variant of Whitney's covering presented in \cite[pp. 167 ff.]{stein:70}. In the definition of $\Omega_k$ set $c=3$, and take the squares obtained there (possibly uniformly 
subdivided into smaller squares to ensure the maximal radius $\delta$) as $\hat{q}$. This easily gives the first and the second assertion.

To estimate the number of neighbours fix a square $q$ and let $\{q_i\}$ be those squares that intersect $q$. Each square $q_i$ is contained in the square with the same center as $q$ and radius 
$2c\ell_{q}$ and so are the disjoint smaller squares $\hat q_i$. Since additionally $\ell_{\hat q_i}{{\geq \ell_{q}/(2c)}}$ we can conclude that their number is uniformly bounded. Property (iii) follows 
immediately, the bound on the distance of neighbours follows from (ii).
\end{proof}

\def\calGII{\calB}
\begin{theorem}[Jump set made of segments]\label{theodensityfull}
 Let $u\in SBV_y^p(\Omega)$, where $\Omega\subset\R^2$ is a Lipschitz bounded set. 
 Then there is a
 sequence $v_j\in SBV_y^p(\Omega)$ with $v_j\to u$ in $L^1$, $u=v_j$ on $\partial\Omega$, $\limsup_j E(v_j,\Omega)\le E(u,\Omega)$, 
 such that $J_{v_j}$ is locally a finite union of segments.\\
 Moreover: If $U \subset \Omega$ open is such that $B_R(U)\cap J_u{{\cap\Omega}}=\emptyset$ for some $R>0$, then the sequence $v_j$ can be constructed in a way that $v_j=u$ {{in}} $U$.
\end{theorem}
We say that  $J_v$ is {{locally}} a finite union of segements, if for any $\omega\subset\subset \Omega$ there is a finite union of segments $S\subset\R^2$
such that ${{J_{v}\cap\omega}}$ coincides, up to $\calH^1$-null sets, with $S\cap \omega$.
\begin{proof}
{\em Step 1: Treatment of the main part of the jump set.}\\
Since $J_u$ is rectifiable there are countably many curves $\gamma_j\in C^1(\R)$ such that
$J_u\subset \cup_i \{(x_1,\gamma_i(x_1): x_1\in\R\}\cup N$, with $\calH^1(N)=0$. All curves 
can be taken as graphs with respect to $x_1$, since for {{almost every}} $x\in J_u$ the normal is $e_2$.

Fix $\eps\in(0,1/4)$.
For $\calH^1$-almost every $x\in J_u$ there is a curve $\gamma_x\in C^1(\R)$ with $\gamma_x'(x_1)=0$ such that
\begin{equation*}
 \lim_{r\to0} \frac{1}{2r} \calH^1((J_u\Delta \gamma_x^*) \cap Q_r(x))=0\,,
\end{equation*}
and 
\begin{equation*}
 \lim_{r\to0} \frac{1}{2r} \|\nabla u\|_{L^p(Q_r(x))}^p =0\,.
\end{equation*}
Here we write $\gamma_x^*$ for the graph $\{(s,\gamma_x(s)): s\in\R\}$ of the curve $\gamma_x$.
Since for almost every $r>0$ one has $\calH^1(J_u\cap \partial Q_r(x))=0$, there is a fine cover of $\calH^1$-almost all of 
$J_u$ with squares $q$ such that
\begin{equation}\label{propgoodsquaretyp3f}
 \calH^1((J_u\Delta\gamma_q^*) \cap q) \le \frac{\eps \ell_{q}}4\,,
 \calH^1(J_u\cap\partial q) =0\,,
\gamma_q^*\cap q\subset R_q^\eps\,,
 \text { and }
 \|\nabla u\|_{L^p(q)}^p\le \eps^{p+1} \ell_{q}.
\end{equation}
Here we denote by $\gamma_q$ the curve pertaining to the midpoint of $q$, and by $R_q^\eps:=x+(-\ell_q,\ell_q)\times(-\eps \ell_q, \eps \ell_q)$ the central stripe of $q$.
The first condition implies $\calH^1(J_u\cap q)\ge (1-\eps/4) \ell_q$.
We can extract a countable set of disjoint squares which cover 
$\calH^1$-almost all of $J_u$, and in particular a finite set $\calG:=\{q_1,\dots, q_M\}$ of disjoint squares
with the property (\ref{propgoodsquaretyp3f}) which cover $J_u$ up to a set of measure $\eps$.

We apply Proposition \ref{proptyp3} to each square $q\in \calG$,  with $\rho:=\eps \ell_q$,
and define $w$ as the result in each $q$, and as $u$ outside.
We obtain $w\in SBV^p_y(\Omega)$ which coincides with $u$ around $\partial\Omega$, and such that
for any $q\in \calG$ the set
$J_w\cap \hat q^{(1-\eps)}$ is the union of two segments.
Here $\hat q^{(1-\eps)}$ is the square with the same center as $q$ and side length $1-\eps$ times smaller.
The jump set is estimated by 
\begin{align*}
\calH^1(J_w\setminus \cup_{q\in\calG} \hat q^{(1-\eps)})\le& \calH^1(J_w\setminus \cup_{q\in\calG} q) + \sum_{q\in\calG} \calH^1(J_w\cap q\setminus \hat q^{(1-\eps)})
\le \eps + \sum_{q\in\calG} c \eps \ell_{q} \\
\le &
\eps + c \eps \calH^1(J_u) 
\end{align*}
 and, since by  Proposition \ref{proptyp3} the total length of $J_w\cap \hat q^{(1-\eps)}$ is bounded by $2(1+\eps)\ell_q$
 and by (\ref{propgoodsquaretyp3f}) $\calH^1(J_u\cap \hat q)\ge 2(1-\eps)\ell_q$, 
\begin{align*}
\calH^1(J_w\cap \cup_{q\in\calG} \hat q^{(1-\eps)})\le& \sum_{q\in\calG} 2(1+\eps) \ell_{q} \le
(1+3\eps) \calH^1(J_u\cap \cup_{q\in\calG}  q)\,.
\end{align*}
For the gradient term we obtain, again using  Proposition \ref{proptyp3} and (\ref{propgoodsquaretyp3f}),
 \begin{align*}
\|\nabla w\|_{L^p(\Omega)}^p\le& \|\nabla u\|_{L^p(\Omega)}^p+ \sum_{q\in\calG}\frac{c}{\eps^p} \|\nabla u\|_{L^p( q)}^p\\
\le
& \|\nabla u\|_{L^p(\Omega)}^p+ \frac{c}{\eps^p} \sum_{q\in\calG} \eps^{p+1} \ell_{q} 
\le
 \|\nabla u\|_{L^p(\Omega)}^p+ c\eps \calH^1(J_u) \,. 
\end{align*}
We also estimate, using the same bound and Lemma \ref{lemmapoincarehoriz} in each square,
\begin{equation*}
 \|w-u\|_{L^p(\Omega)}^p\le \sum_{q\in\calG} c \ell_q^p\|\nabla u-\nabla w\|_{L^p(q)}^p
 \le   \sum_{q\in\calG} c\eps^{p+1} \ell_q^{p+1}\le c \eps^{p+1}|\Omega| (\diam\Omega)^{p-1}\,. 
\end{equation*}
This concludes the treatment of the main part of the jump set.

We summarize what we have obtained so far. Given $\eps>0$, we obtained a finite set of squares $\calG$ 
and $w\in SBV^p_y(\Omega)$ such that $J_w\cap \hat q^{(1-\eps)}$ consists of two segments for any $q\in\calG$,
$\|\nabla w\|_{L^p(\Omega)}\le \|\nabla u\|_{L^p(\Omega)}+ M_u\eps$, 
$\calH^1(J_w\cap \Omega)\le \calH^1(J_u\cap\Omega)+M_u\eps$,  
 $u=w$ around $\partial\Omega$,
 $\|u-w\|_{L^1(\Omega)}\le M_u\eps$, and $\calH^1(J_u\setminus\omega)\le M_u \eps$. Here $M_u$ is a constant that 
 may depend on $u$ and $\Omega$ but not on $\eps$ and we let  $\omega:=\cup_{q\in\calG} \hat q^{(1-\eps)}$ be the union of the smaller squares.

{\em Step 2: Treatment of the small part of the jump set.}\\
 We intend to cover $\Omega\setminus\omega$ with squares
much smaller than those composing $\omega$.
We fix $\delta>0$, chosen below.
We choose $N$ families of squares $\calF_1'$, \dots, $\calF_N'$ which cover $\Omega$
as in Lemma \ref{lemmacover}.
Inside $\omega$ we do not need to modify the function $w$ any more, hence the squares in $\calG':=\{q\in\calF': q\subset\omega\}$ need not be touched.
We set $\calF_k:=\calF_k'\setminus\calG'$, and correspondingly 
$\calF:=\calF'\setminus\calG'=\cup_k\calF_k$. Property (i) now reads $\Omega=\omega\cup \bigcup_{q\in\calF} q$, (ii) and (iii) still hold, the estimates on $\calN_k$ are also still valid.
Since $J_w$ is union of two segments inside each of the squares in $\calG$, we have
\begin{equation*}
 \calH^1(J_w\cap \bigcup_{q\in\calF} q) \le \calH^1(J_u\cap \Omega\setminus\omega) + 2\delta \#\calG 
 \le 2M_u\eps
\end{equation*}
provided $\delta$ is chosen sufficiently small (on a scale set by $\eps$ and $\calG$, which in turn depends on $\eps$ {and $u$}).

Fix $\eta\in(0,{{10^{-2}}})$. Let $\calGII:=\{q\in \calF: \calH^1(J_w\cap q) \ge \eta \ell_q\}$ be the set of ``bad'' squares, 
on which we cannot use Proposition \ref{proptyp1}.

We intend to iteratively apply the constructions in the individual squares of each family.
We set $z_0:=w$. 

We first explain how to construct $z_k$ from $z_{k-1}$, working on the (disjoint) squares of the subfamily $\calF_k$.
We describe the procedure at step $k$, dropping the index from the notation for simplicity. 
We set $y_0:=z_{k-1}$, let $(q_m)_{m\in\N}$ be an enumeration of $\calF_k$, and define for $m\in\N$ the function
$y_m$ by $y_m=y_{m-1}$ on $\Omega\setminus q_m$, and inside $q_m$ as the result of Proposition \ref{proptyp1}  with $\rho:=\eta^{1/2}\ell_{q_m}$ if $q_m\not\in \calGII$, and the result
of  Proposition \ref{proptyp2} if $q_m\in \calGII$. 
We obtain
\begin{equation*}
 \calH^1(J_{y_m})\le \calH^1(J_{y_{m-1}}) + c \calH^1(J_{y_{m-1}}\cap q_m).
\end{equation*}
Since $y_{m-1}=y_0$ in $q_m$ and the squares are disjoint, this gives $ \calH^1(J_{y_m}) \le (1+c) \calH^1(J_{y_0})$ independently of $m$. Further, 
the constructions give
\begin{equation*}
 |Dy_m-Dy_{m-1}|(\Omega)= |Dy_m-Dy_{m-1}|(q_m)\le c |Dy_{m-1}|(q_m)=c |Dy_0|(q_m)\,,
\end{equation*}
where as above we used that $y_{m-1}=y_0$ on $q_m$. Since the $q_m$ are disjoint, this shows that $\sum_m |Dy_m-Dy_{m-1}|(\Omega)\le
c|Dy_0|(\Omega)<\infty$, hence
(recalling $y_m=y_0$ on $\partial\Omega$)
the sequence $y_m$ is a Cauchy sequence in $BV$. Let $y_\infty$ be the limit. Since all $y_m$ have the same traces on $\partial\Omega$
as $y_0$, so does $y_\infty$. 
At the same time,  since 
$\eta\le \eta^{1/p'}$, 
\begin{equation}\label{eqnablawmqmf}
 \|\nabla y_m\|_{L^p(q_m)}^p\le \|\nabla y_{m-1}\|_{L^p(q_m)}^p+c\eta^{p/(2p')} \|\nabla y_{m-1}\|_{L^p(q_m)}^p + c  \|\nabla y_{m-1}\|_{L^p(q_m)}^p \chi_{q_m\in\calGII}
\end{equation}
where the last term only appears if $q_m$ is in $\calGII$.
Again, since $y_{m-1}=y_0$ on $q_m$ and the squares are disjoint we obtain  a uniform bound on $ \|\nabla y_m\|_{L^p(\Omega)}$,
\begin{equation}\label{eqestnablaymp}
 \|\nabla y_m\|_{L^p(\Omega)}^p\le (1+c\eta^{p/(2p')})\|\nabla y_0\|_{L^p(\Omega)}^p+ c \sum_{q\in\calF_k\cap \calGII}  \|\nabla y_{0}\|_{L^p(q)}^p \,.
\end{equation}
By the $SBV^p$ compactness theorem, the limit $y_\infty$ belongs to $SBV^p(\Omega)$. 
 We define $z_k$ as $y_\infty$. It is clear that $z_k\in SBV^p_y(\Omega)$ with $z_k=u$ on $\partial\Omega$. Further, using Lemma \ref{lemmapoincarehoriz}
 \begin{equation}\label{eqzkzk1lp}
  \|z_k-z_{k-1}\|_{L^p(\Omega)}^p \le \sum_m c \ell_{q_m}^p \|\nabla y_m\|^p_{L^p(q_m)} 
   {{\le  c \delta^p \|\nabla z_{k-1}\|^p_{L^p(q_m)} .}}
 \end{equation}
 Iterating this procedure over the $N$ families we obtain the function $z_N$. 
 By construction, $z_N\in SBV^p_y(\Omega)$ and $z_N=u$ on $\partial\Omega$. Further, any open set $\hat\Omega\subset\subset\Omega$ intersects only finitely many squares,
 therefore $J_{z_N}\cap \hat\Omega$ is a finite union of segments. 
 It remains to estimate the norms. We first observe that, with $\Omega':=\cup_{q\in\calF} q$, 
\begin{equation*}
 \calH^1(J_{z_k}\cap\Omega')\le c \calH^1(J_{z_{k-1}}\cap\Omega') 
\end{equation*}
which immediately gives $\calH^1(J_{z_N}\cap\Omega')\le c^N \calH^1(J_{w}\cap\Omega')\le 2c^NM_u\eps$. Further, by (\ref{eqestnablaymp})
\begin{equation}\label{eq:B}
\|\nabla z_k\|_{L^p(\Omega)}^p \le
(1+c\eta^{p/(2p')}) \|\nabla z_{k-1}\|_{L^p(\Omega)  }^p
+c \sum_{q\in \calF_k\cap \calGII}\|\nabla z_{k-1}\|_{L^p(q)  }^p\,.
\end{equation}
In order to estimate the last term, we use the second part of Lemma \ref{lemmacover}.
The key property of the construction we use here is the fact that at each step $k$ the function is only modified in the squares of $\calF_k$, which are disjoint. 
Let now $q\in\calF$ be a generic square. Since only the changes in squares $q'\in\calF_k$ which 
intersect $q$ modify the function inside $q$, we obtain from (\ref{eqnablawmqmf})
\begin{equation*}
 \|\nabla z_k\|_{L^p(q)} \le c\sum_{ q'\in \calF_k\cap \calN_1(q)} \|\nabla z_{k-1}\|_{L^p(q')}\,.
\end{equation*}
Iterating this condition, and recalling the properties in  Lemma \ref{lemmacover}, {{we have}}
\begin{equation*}
 \|\nabla z_k\|_{L^p(q)} \le c^k \sum_{ q'\in \calN_k(q)} \|\nabla w\|_{L^p(q')}.
\end{equation*}
Since any square $q'$ can be reached only starting from its $N$-neighbours, the combinatorial coefficient is uniformly bounded.
We conclude that
\begin{equation*}
 \sum_k \sum_{q\in \calF_k\cap \calGII}\|\nabla z_{k-1}\|_{L^p(q)  }
 \le c \|\nabla w\|_{L^p(\hat\omega)}
\end{equation*}
where 
\begin{equation*}
 \hat\omega:=  \bigcup_{q\in \calF\cap \calGII} \bigcup_{q'\in \calN_N(q)} q'\,.
\end{equation*}
It remains to estimate the size of the set $\hat\omega$. Since $\ell_{q'}\le a^N \ell_q$, we have
\begin{align*}
| \hat\omega|&\le   4 \sum_{q\in \calF\cap \calGII} (a^{{2N}} \ell_q^2) \# \calN_N(q)
\le 4{{a^{2N}b^N}}  \sum_{q\in \calF\cap \calGII} \ell_q^2\\
& \le 4{{\frac{a^{2N}b^N}{\eta}}}  \sum_{q\in \calF\cap \calGII} \delta\,\calH^1({{J_w}}\cap q)\le 4{{\frac{a^{2N}b^N \delta}{\eta}}}  \calH^1({{J_w}}\cap\Omega') 
\le {{\frac{8M_ua^{2N}b^N\eps\delta}{\eta}}}.
\end{align*}
At this point we choose $\eta:=\eps$.  In the limit $\delta\to0$ we have $|\hat\omega|\to0$, therefore if $\delta$ is sufficiently small 
we have $\|\nabla w\|_{L^p(\hat\omega)}\le\eps${{, and so \eqref{eq:B} yields $\|\nabla z_N\|_{L^p(\Omega)}^p\leq (1+C\eps^{p/(2p')})\|\nabla w\|_{L^p(\Omega)}^p+\eps$}}. Further, for $\delta$ sufficiently small iterating
(\ref{eqzkzk1lp}) for $k=1,\dots, N$ yields $\|z_N-w\|_{L^p}\le\eps$.

Finally, we define $v_j$ as the function $z_N$ obtained with $\eps:=1/j$. 

{\em Step 3: Inclusion of the condition on $U$.}\\
If there are $U\subset \Omega $ open and $R>0$ such that $J_u\cap B_R(U)= \emptyset$ we modify the construction slightly. Let first the maximal diameter of all squares (in both iterations) be smaller 
then $R/2$ (this means $2\sqrt2 \ell_q< R/2$ for all squares $q$). Further, at the beginning of Step 2 we also exclude 
all squares which contain no jump set (and in which, therefore, no action is needed). Precisely,
we replace $\calG'$ by
 $\calG'':=\{q\in\calF': q\subset\omega \text{ or } \calH^1(q\cap J_w)=0\}$.
Then none of the ``surviving'' squares touches $U$, hence $u$ is not modified in $U$. All other properties still hold.
\end{proof}

\section*{Acknowledgments} 
We thank Hans Kn\"upfer for helpful discussions.  
This work was partially supported 
by the Deutsche Forschungsgemeinschaft through the Sonderforschungsbereich 1060 
{\sl ``The mathematics of emergent effects''}, project A6. 
\appendix
\section{Scaling Law}
In this appendix we prove Theorem \ref{th:scaling}. 
We follow the paths of the proof of the scaling law for the special case $p=2$ (see \cite{kohn-mueller:92,kohn-mueller:94,conti:06,diermeier:10,zwicknagl:14}).
\begin{proof}[Proof of  Theorem \ref{th:scaling}]
{\em Step 1: Upper bound}\\
Note that the constant function $u_c:=0$ yields $I_{\theta,\eps}^p(u_c)=\theta^p$. 
If $\eps\le \theta^p$ we  construct a test function $u_b$ with $I_{\theta,\eps}^p(u_b)\leq c\theta^p(\eps/\theta^p)^{p/(p+1)}$, using the 
variant of the branching construction from \cite{kohn-mueller:94} given in \cite{zwicknagl:14} in the formulation of \cite{conti-zwicknagl:15}. The construction given in \cite[Lemma 
5.2]{conti-zwicknagl:15} shows that for arbitrary $\ell>0$ and $h>0$ there exists a function $b=b^{(\ell,h)}:(0,\ell)\times\R\to\R$ with the following properties: 
\begin{itemize}
\item[(i)] $b(x_1,0)=0$, $b(x_1,\cdot)$ is $h$-periodic,
\[b(\ell,x_2)=\begin{cases}
-\theta x_2&\text{\quad if\ }0\leq x_2\leq h(1-\theta)/2,\\
(1-\theta)(x_2-h/2)&\text{\quad if\ } h(1-\theta))/2\leq x_2\leq h(1+\theta)/2,\\
-\theta x_2+\theta h&\text{\quad if\ } h(1+\theta)/2\leq x_2\leq h,\\
\end{cases} \]
\item[(ii)] $b(0,x_2)=\frac{1}{2}b(\ell,2x_2)$,
\item[(iii)] $\|\partial_1b\|_{L^p((0,\ell)\times(0,h))}\leq c\frac{\theta^ph^{p+1}}{\ell^{p-1}}$,
\item[(iv)] $\int_{(0,\ell)\times(0,h)}|D^2b|\leq C(\ell+{{\theta}}h)$,
\item[(v)] $\partial_2b\in\{-\theta,1-\theta\}$ almost everywhere.
\end{itemize} 
We now proceed as in \cite[Lemma 1]{kohn-mueller:94}. We choose a refining parameter $\alpha\in(2^{-p/(p-1)},2^{-1})$, and choose $N\in\N$ such that $N\sim (\theta^p/\eps)^{1/(p+1)}$.  We  decompose 
$(0,1)\times(0,1)$ into rectangles 
\[R_{ij}:=(\alpha^{i+1},\alpha^i)\times (\frac{j}{2^iN},\frac{{{j+1}}}{2^iN}),\qquad i=0,1,\dots\text{\ and \ }j=0,\dots,{{2^iN}}-1.\]
For $i\leq I$ with $(2\alpha)^I\sim \theta/N$, we set $u_b(x_1,x_2)=b^{(\ell_i,h_i)}(x_1-\alpha^{i+1},x_2)$ on $R_{i0}$, where $\ell_i:=(1-\alpha)\alpha^{ i}$ and $h_i:=1/(2^iN)$. The function $u_b$ is then 
extended $1/(2^iN)$-periodically in $x_2$-direction to the remaining $R_{ij}$. Finally, we use linear interpolation in $x_1$ on $(0,\alpha^{I+1})\times(0,1)$. The total energy in 
$(\alpha^{I+1},1)\times(0,1)$ is then estimated by
\[C\sum_{i=1}^I \left(\frac{\theta^p}{(2^iN)^{p+1}\alpha^{i(p-1)}}2^iN+\eps\alpha^{i}2^iN\right)\leq C(\eps\theta)^{p/(p+1)}.\]
Since $\alpha <(2\alpha)^p$, 
$\alpha^I\le (2\alpha)^{pI}\sim (\theta/N)^p \sim (\eps\theta)^{p/(p+1)}$ and 
the transition layer obeys the analogue upper bound. 

{\em Step 2: Lower bound}\\
To derive the lower bound, we follow closely the lines of \cite[Proof of Theorem 1]{zwicknagl:14}, which in turn is based on \cite{conti:06}. Let $\theta$, $\eps$, $p$ be given as in Theorem \ref{th:scaling}, and fix $u\in\mathcal{B}$ arbitrary. 
Set $t:=\min\{1,(\eps/\theta^p)^{1/(p+1)}\}$. Choose $J:=[y,y+t]\subset(0,1)$ such that
\[I_J(u):= \int_{(0,1)\times J}|\partial_1u|^p+\min\left\{|\partial_2u+\theta|^p,\ |\partial_2 u-(1-\theta)|^p\right\}\dxy+\eps|D^2u|((0,1)\times J)\leq 2tI_{\theta,\eps}^p(u).\]
By Fubini, there exists $M\subset (0,1)$ with $\calL^1(M)>0$ such that for all $x_1\in M$
\begin{eqnarray*}
&&\int_{\{x_1\}\times J}|\partial_1u|^p+\min\left\{|\partial_2u+\theta|^p,\ |\partial_2 u-(1-\theta)|^p\right\}\dy\\
&&\leq3\int_{(0,1)\times J}|\partial_1u|^p+\min\left\{|\partial_2u+\theta|^p,\ |\partial_2 u-(1-\theta)|^p\right\}\dxy 
\end{eqnarray*}
and 
\[|\partial_2\partial_2 u|(\{x_1\}\times J)\leq 3|D^2v|((0,1)\times J). \]
We decompose $M:=M_1\cup M_2\cup M_3$, where 
\begin{eqnarray*}
M_1&:=&\{x_1\in M: |\partial_2u+\theta|\leq |\partial_2 u-(1-\theta)|\text{ a.e. }x_2\in J\},\\
M_2&:=&\{x_1\in M: |\partial_2u+\theta|\geq |\partial_2 u-(1-\theta)|\text{ a.e. }x_2\in J\},\\
M_3&:=&M\setminus (M_1\cup M_2).
\end{eqnarray*}
One of the three sets has positive measure. 
If $M_1$ has positive measure, then fix $x_1\in M_1$. We use the following variant of an estimate from 
\cite[Lemma 1]{conti:06} (see \cite{zwicknagl:14})
\[t^2\theta\lesssim\min_{c\in\R}\|u(x_1,x_2)+\theta x_2-c\|_{L^1(J)} +\|u(x_1,\cdot)\|_{L^1(J)}.\]
By Poincar\'{e}'s inequality and definition of $M_1$ (recall that $p':=p/(p-1)$),
\begin{eqnarray*}
\min_{c\in\R}\|u(x_1,x_2)+\theta x_2-c\|_{L^1(J)} \le t\|\partial_2 u+\theta\|_{L^1(J)}\le t^{1+1/p'}(I_J(u))^{1/p},
\end{eqnarray*}
and by the boundary conditions and the fundamental theorem of calculus,
\begin{eqnarray*}
\|u(x_1,\cdot)\|_{L^1(J)}\leq\|\partial_1u\|_{L^1((0,1)\times J)}\leq t^{1/p'}(I_J(u))^{1/p}.
\end{eqnarray*}
Therefore, in this case, 
\begin{eqnarray*}
I_{\theta,\eps}^p(u)\gtrsim \theta^p\min\{t^p,1\} \geq \theta^p\min\{1,(\eps/\theta^p)^{p/(p+1)}\}.
\end{eqnarray*}
Similarly, if $M_2$ has positive measure, we obtain the larger lower bound $I(u)\gtrsim \min\{t^p,1\}$. Finally, if $M_3$ has positive measure, fix $x\in M_3$. There are two possibilities: Either 
\[\min\{|\partial_2u+\theta|,|\partial_2u-(1-\theta)|\}\geq 1/4\text{\qquad for a.e.\ }x_2\in J, \]
which implies
\[\frac{t}{4^p}\leq\int_{\{x_1\}\times J} \min\left\{|\partial_2v+\theta|^p,\ |\partial_2 v-(1-\theta)|^p\right\}\dy\leq I_J(u),\]
or $|\partial_2\partial_2u|(\{x_1\}\times J)\geq 1/4$. Hence, 
\[I_{\theta,\eps}^p(u)\gtrsim\min\{\eps/t,\,1\}\gtrsim\theta^p\min\{(\eps/\theta^p)^{p/(p+1)},\,1\}. \]
 \end{proof}

\bibliographystyle{siam}
\bibliography{biblio}

\end{document}